\documentclass[a4paper,10pt, leqno]{amsart}
\date{September 10, 2020}
\usepackage{amsmath, amssymb, graphics, setspace}
\usepackage[active]{srcltx}
\usepackage[dvips]{graphicx}
\usepackage[usenames]{color}

\title[Curved Foldings]
{%
On the existence of four or more curved foldings with common 
creases  and crease patterns}

\author{A. Honda}
\address[Atsufumi Honda]{
Department of Applied Mathematics, 
Faculty of Engineering, Yokohama National University,
79-5 Tokiwadai, Hodogaya, Yokohama 240-8501, Japan
}
\email{honda-atsufumi-kp@ynu.ac.jp}

\author{K.~Naokawa}
\address[Kosuke Naokawa]{%
Department of Computer Science, 
Faculty of Applied Information Science,
Hiroshima Institute of Technology,  
2-1-1 Miyake, Saeki, Hiroshima, 731-5193, Japan
}
\email{k.naokawa.ec@cc.it-hiroshima.ac.jp}

\author{K.~Saji}
\address[Kentaro Saji]{%
  Department of Mathematics,
  Faculty of Science,
  Kobe University,
  Rokko, Kobe 657-8501}
\email{saji@math.kobe-u.ac.jp}

\author{M.~Umehara}
\address[Masaaki Umehara]{%
  Department of Mathematical and Computing Sciences,
  Tokyo Institute of Technology,
  Tokyo 152-8552, Japan}
\email{umehara@is.titech.ac.jp}

\author{K.~Yamada}
\address[Kotaro Yamada]{%
  Department of Mathematics,
  Tokyo Institute of Technology,
  Tokyo 152-8551, Japan}
\email{kotaro@math.titech.ac.jp}

\newcommand{\op}[1]{{\operatorname{#1}}}

\newcommand{\R}{\boldsymbol{R}}

\newcommand{\Z}{\boldsymbol{Z}}
\newcommand{\Q}{\boldsymbol{Q}}

\newcommand{\mc}[1]{{\mathcal #1}}
\newcommand{\mb}[1]{{\mathbf #1}}

\newcommand{\pmt}[1]{{\begin{pmatrix} #1  \end{pmatrix}}}

\renewcommand{\phi}{\varphi}
\renewcommand{\epsilon}{\varepsilon}
\newcommand{\dy}{\displaystyle}

\renewcommand{\det}{\op{det}}

\numberwithin{equation}{section}
\newtheorem{Theorem}{Theorem}[section]
\newtheorem{Proposition}[Theorem]{Proposition}
\newtheorem{Prop}[Theorem]{Proposition}
\newtheorem{Corollary}[Theorem]{Corollary}
\newtheorem{Cor}[Theorem]{Corollary}
\newtheorem{Lemma}[Theorem]{Lemma}
\newtheorem{Fact}[Theorem]{Fact}
\theoremstyle{definition}
\newtheorem{Def}[Theorem]{Definition}
\newtheorem{Definition}[Theorem]{Definition}
\newtheorem{Remark}[Theorem]{Remark}
\newtheorem{Example}[Theorem]{Example}
\newtheorem*{acknowledgments}{Acknowledgments}


\thanks{%
The first author was partially supported by 
Grant-in-Aid for Early-Career Scientists
 No.~19K14526 and 
Grant-in-Aid for 
Scientific Research (B) No.~20H01801. 
The second author was partially supported by 
Grant-in-Aid for Young Scientists (B) No.~17K14197,
and the third author
was 
partially supported by  Grant-in-Aid for 
Scientific Research (C) No.\ 18K03301
The fifth author 
was partially 
supported by Grant-in-Aid for 
Scientific Research (B) No.\ 17H02839.
}%
\begin{document}
\maketitle
\begin{abstract}
Consider an oriented curve $\Gamma$ in a domain $D$ in the plane $\R^2$.
Thinking of $D$ as a piece of paper, one can make a curved 
folding in the Euclidean space $\R^3$.
This can be expressed as the image of an 
\lq\lq origami map\rq\rq\ 
$\Phi:D\to \R^3$ such that $\Gamma$ is the 
singular set of $\Phi$, the word \lq\lq origami\rq\rq\ 
coming from the Japanese term for paper folding.
We call the singular set image $C:=\Phi(\Gamma)$ the {\it crease} of $\Phi$
and the singular set $\Gamma$ the {\it crease pattern} of $\Phi$.

We are interested in the number of origami 
maps whose creases and 
crease patterns are $C$ and $\Gamma$, respectively.
Two such possibilities have been known.
In the authors' previous work, two other new possibilities and 
an explicit example with four such non-congruent 
distinct curved foldings were established.

In this paper, we determine the possible values for 
the number $N$ of congruence classes of curved foldings
with the same crease and crease pattern.
As a consequence, if $C$ is a non-closed simple arc,
then $N=4$ if and only if both $\Gamma$ and $C$ do not admit any symmetries.
On the other hand, when $C$ is a closed curve, 
there are infinitely many distinct possibilities 
for curved foldings with the same crease and crease pattern, in general.
\end{abstract}

\begin{figure}[htb]
\begin{center}
 \begin{center}
  \begin{tabular}{ccc}
 \raisebox{0.1in}{%
 \includegraphics[width=0.3\linewidth]{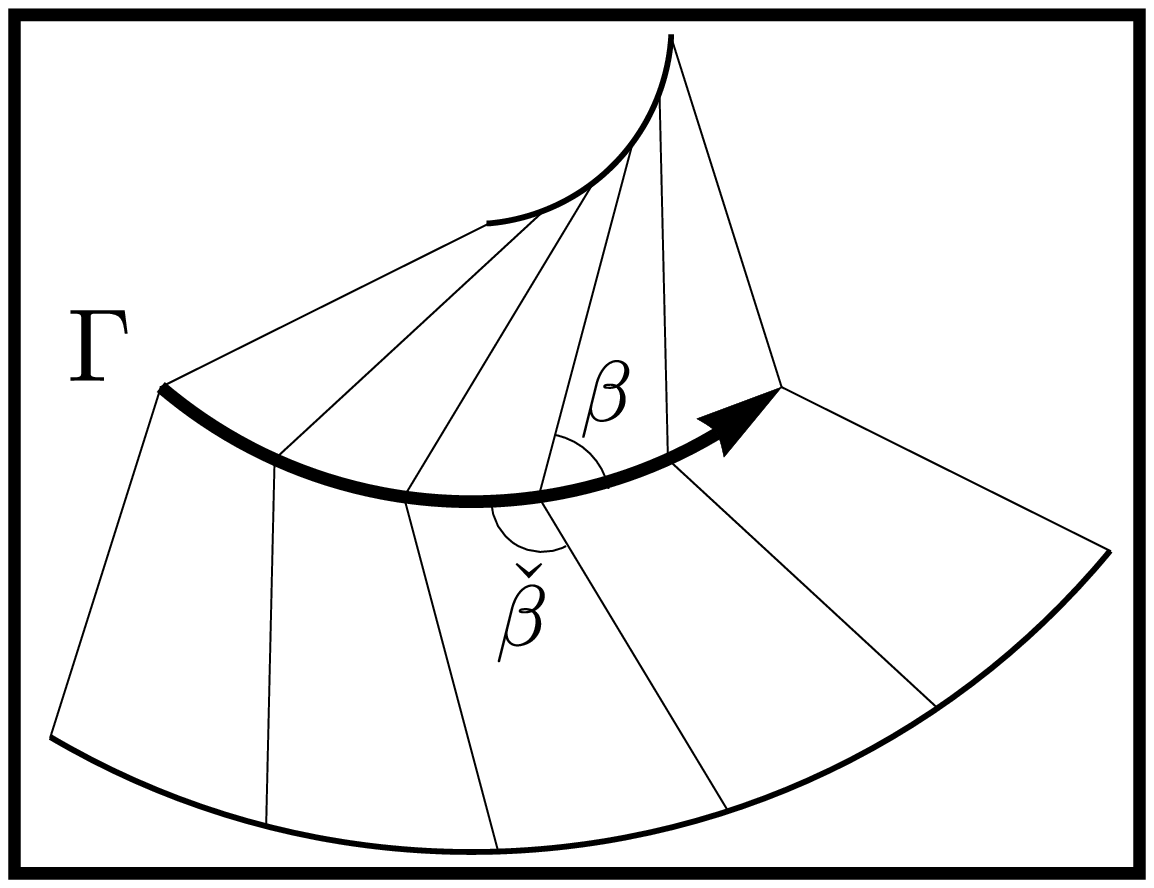}} & \phantom{aaa}
 \raisebox{0.7in}{\includegraphics[width=0.08\linewidth]{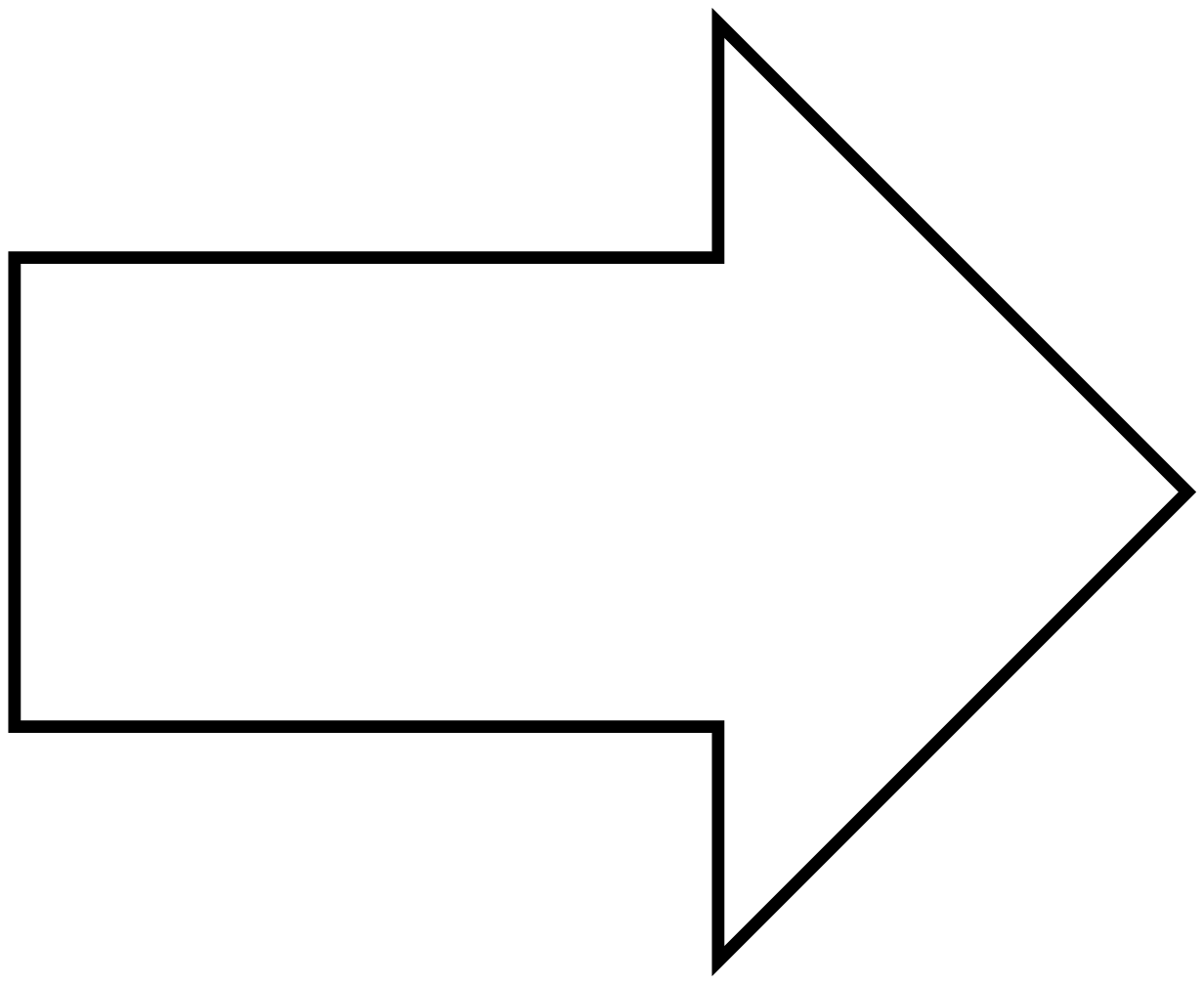}} &
 \includegraphics[width=0.35\linewidth]{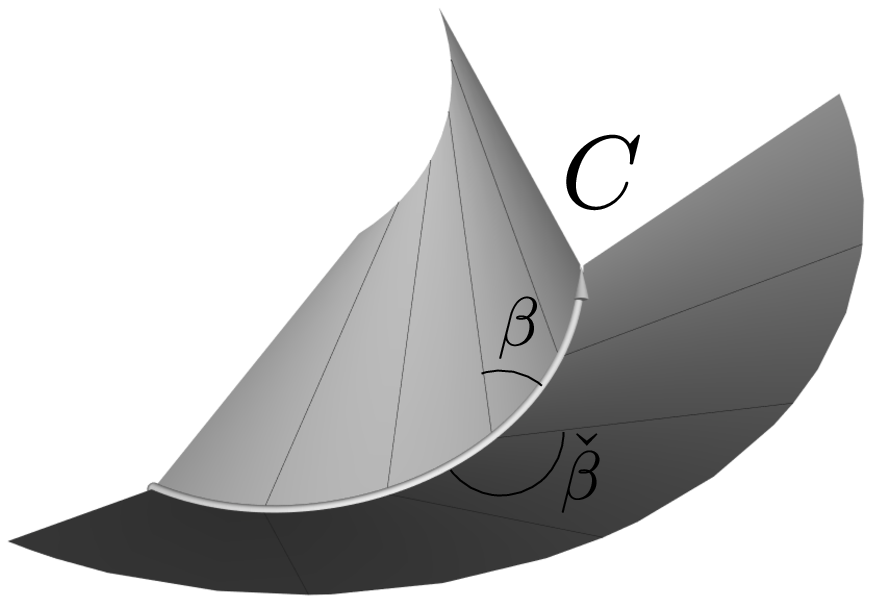}
 \end{tabular}
 \end{center}
\caption{A given crease pattern $\Gamma$  (left)
and its realization (right) 
as a curved folding along the crease $C$.}\label{fig:A}
\end{center}
\end{figure}

\section*{Introduction}
The geometry of curved foldings is, nowadays, 
an important subject not only
from the viewpoint of mathematics but also from the viewpoint of engineering.
Works on this topic have been published by numerous authors; for example,
\cite{DDHKT1,DDHKT2, DD, FT0, Hu, K} and \cite[Lecture 15]{FT1} 
are fundamental references.

Drawing a curve $\Gamma$ in $\R^2$, we think of
a tubular neighborhood of $\Gamma$ as a piece of paper.
Then we can fold this along $\Gamma$, and obtain a 
curved folding in the Euclidean space $\R^3$, 
which is a developable surface
(as a subset of $\R^3$)
whose singular set image is a space curve $C\,(\subset \R^3)$.
In this paper, we fix an orientation of $C$ and denote by $-C$
the image of the same curve with the opposite orientation.  
We denote by $|C|$ without considering its 
orientation.  

In this paper, we focus on curved foldings 
which are produced from a single curve $\Gamma$ in $\R^2$
satisfying the following properties: 
\begin{enumerate}
\item[(i)] 
The length of $\Gamma$ is equal to that of $C$
 (\cite[Page 29 (5)]{FT0}). Moreover, the two curves have 
a bijective correspondence by an arc-length parametrization.
\item[(ii)] The curvature functions of $C$
and $\Gamma$ 
have no zeros (\cite[Page 29 (3)]{FT0}).
\item[(iii)] The curves $C$ and 
$\Gamma$ have no self-intersections.
\item[(iv)] The absolute value of the curvature function of 
$\Gamma$ is less than the curvature function of $C$
at each point of $\Gamma$ (\cite[Page 28 (1)]{FT0}).
\end{enumerate}
If a curved folding satisfies (i)--(iii) and 
the following condition
(stronger than (iv)), then 
it is said to be {\it admissible} 
(cf. \eqref{eq:alpha_f2} and also \eqref{eq:alpha_f20}).
\begin{enumerate}
\item[(iv${}'$)]  
The maximum of the absolute value of the curvature function of 
$\Gamma$ is less than the minimum of the curvature function of 
$C$ (\cite[Page 28 (1)]{FT0}).
\end{enumerate}
For a pair $(\Gamma,|C|)$ giving a curved folding $P$,
there is another possibility for corresponding 
curved foldings (see \cite{FT0}). 
Moreover, in the authors' previous work \cite{Ori},
two additional possibilities were found when $P$ is admissible, 
and an explicit example of four non-congruent curved foldings
with the same crease and crease pattern was given.
The purpose of this paper is to further develop the
discussions in \cite{Ori}.
In fact, we are interested in the number of congruence 
classes of curved foldings with
a given pair of crease $|C|$ and crease pattern $\Gamma$.
Since this number is closely related to the 
symmetries of $|C|(\subset \R^3)$ and $\Gamma(\subset \R^2)$, 
we give the following:

\begin{Def}\label{def:S}
A subset $A$ of the Euclidean space $\R^k$ ($k=2,3$) 
is said to have a {\it symmetry} if 
$T(A)=A$ holds for an isometry $T$ of $\R^k$ which is not the 
identity map, and $T$ is called a {\it symmetry} of $A$.
Moreover, if there is a point $\mb x\in A$ such that 
$T(\mb x)\ne \mb x$, then $A$ is said to have a 
{\it non-trivial symmetry} $T$.
On the other hand, a symmetry $T$ of $A$ is called {\it positive}
 (resp. {\it negative}) if $T$ is an orientation preserving 
(resp. reversing) isometry of $\R^k$.
\end{Def}

We denote by $\mathcal P(\Gamma,|C|)$  
(resp. $\mathcal P_*(\Gamma,|C|)$) 
the set of curved foldings (resp. the set of admissible curved foldings)
whose creases and crease patterns are $C$ and $\Gamma$
satisfying (i)--(iv) (resp. (i)--(iii) and (iv${}'$)),
see \eqref{eq:o2} for the precise definition.
We prove the following, which is a refinement of 
\cite[Theorems A and B]{Ori}:

\medskip
\noindent
{\bf Theorem A.}
{\it Let $C$ be the image of an embedded curve which is defined on a bounded closed interval. Then the number $n$ of the elements in $\mathcal P_*(\Gamma,|C|)$ as subsets in $\R^3$
is four if $\Gamma$ has no symmetries. 
Otherwise, $n$ is equal to two.}

\medskip
\noindent
{\bf  Theorem B.}
{\it Let $C$ be the image of an embedded curve 
which is defined on a bounded closed interval. 
Then the number $N$ of the congruence classes of
curved foldings in $\mathcal P_*(\Gamma,|C|)$ 
satisfies the following:
\begin{enumerate}
\item if $C$ has no symmetries and 
$\Gamma$ also has no symmetries, 
then $N=4$, 
\item if not the case in {\rm (1)}, then $N\le 2$ holds, 
and \item $N=1$ if and only if
\begin{enumerate}
\item $C$ lies in a plane and has a non-trivial symmetry,
\item $C$ lies in a plane and $\Gamma$ has a symmetry, or
\item $C$ does not lie in any plane and
has a positive symmetry, and $\Gamma$ also has a symmetry.
\end{enumerate}
\end{enumerate}}

\medskip
In \cite{Ori}, an example of $\mc P_*(\Gamma,|C|)$ consisting of four non-congruent subsets in $\R^3$ was given by computing the mean curvature functions along $C$.
However, this approach seems insufficient to prove Theorem B
(see Proposition \ref{prop:H} in Section 4). 
So, in this paper, we will prepare several new 
techniques for its proof (see Sections 2, 3 and 4).

\begin{figure}[hbt]%
 \begin{center}
  \begin{tabular}{c@{\hspace{5em}}c}
       \includegraphics[width=3.0cm]{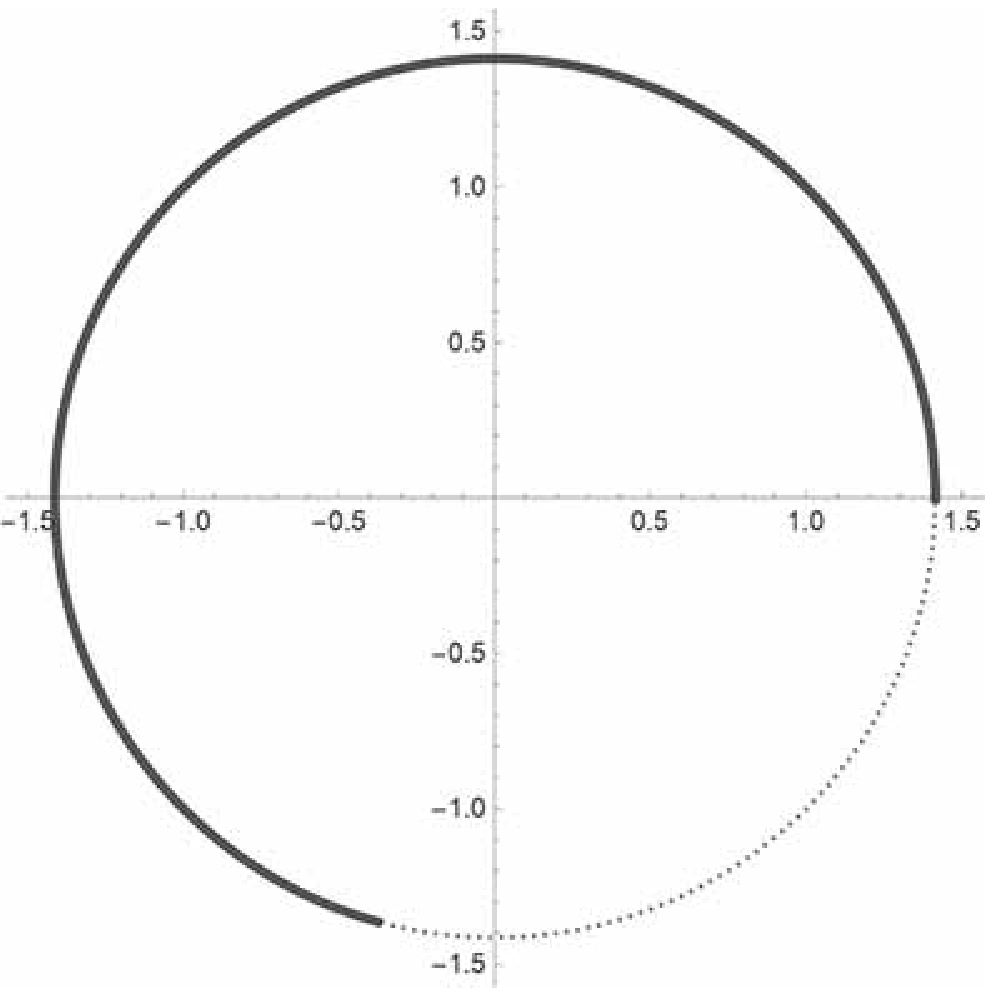} & 
       \includegraphics[width=3.3cm]{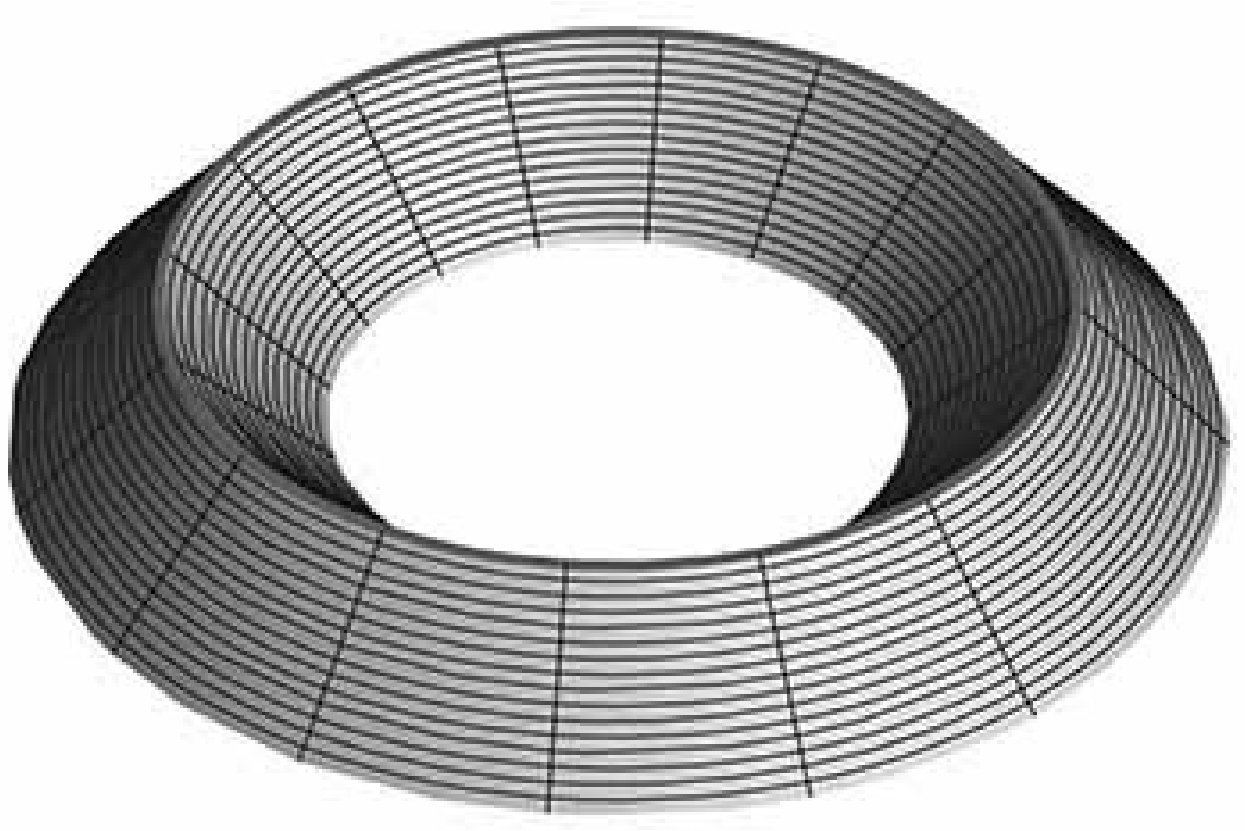} 
  \end{tabular}
 \end{center}
\caption{ 
A crease pattern (left) corresponding to a
curved folding along a circle (right).
}%
\label{CP}
\end{figure}

We next consider the case that \begin{itemize}
\item $C$ is a knot (i.e.~a simple closed curve)
 of length $l(>0)$ 
in $\R^3$ giving a crease of a curved 
folding $P\in \mc P_*(\Gamma,|C|)$, and 
\item $\Gamma$ is a curve of length $l$ embedded in $\R^2$ as a crease pattern. 
\end{itemize}
Even when $C$ is closed, $\Gamma$ may not be a closed curve in general.
In fact, if we consider the curved folding along the unit circle
whose first angular function 
(see Section 1 for the definition)
is $\pi/4$ as in Figure \ref{CP},
then its crease pattern is the sector of the
circle of radius $\sqrt{2}$ whose length is $2\pi$.

We return to the general setting.
Let $\gamma(s)$ be an arc-length 
parametrization of $\Gamma$.
If $C$ is a closed curve of length $l$,
then the curvature function $\mu(s)$ of $\gamma(s)$ 
can be extended as an $l$-periodic function $\tilde \mu(s)$ $(s\in \R)$.
A curved folding $P\in \mc P_*(C,\Gamma)$
consists of a union of two strips along $C$
whose geodesic curvature functions coincide with $\mu(s)$.

\medskip
\noindent
{\bf Theorem C.}
{\it Let $\mb c:\R\to \R^3$ and $\tilde \gamma:\R \to \R^2$ 
be regular curves parametrized by arc-length such that 
\begin{enumerate}
\item $\mb c(s)=\mb c(s+l)$,
\item the curvature function $\tilde \mu:\R\to \R$ of $\tilde \gamma$
induces
a function $\mu:\R/l\Z\to \R$ 
 defined on the one dimensional torus $\R/l\Z$
satisfying
$$
0<\max_{w\in \R/l\Z} \mu(w)<\kappa(s) \qquad (s\in [0,l)),
$$  
where $\kappa(s)$ is the curvature function of $\mb c(s)$.
\end{enumerate}
We set $C:=\mb c([0,l])$ and $\Gamma:=\gamma([0,l])$.
Then there exist four continuous families 
$
\{P^i_{\mb x}\}_{\mb x\in C}
$
$(i=1,2,3,4)$ 
of curved foldings in $\mc P_*(\Gamma,|C|)$
satisfying the following properties:
\begin{enumerate}
\item[(a)]
The set $\mc P_*(\Gamma,|C|)$ coincides with 
$
\displaystyle\bigcup_{\mb x\in C} \{P^1_{\mb x},\,\, P^2_{\mb x},
\,\,P^3_{\mb x},\,\,P^4_{\mb x}\}.
$
\item[(b)] Suppose that $C$ is not a circle and $\Gamma$
is not a subset of a circle. 
Then, for each $P^i_{\mb x}$ $(i\in \{1,2,3,4\},\,\, \mb x\in C)$, 
its congruence class
$$
\Lambda^i_{\mb x}:=\{Q\in \mc P_*(\Gamma,|C|)\,;\, 
\mbox{$Q$ is congruent to $P^i_{\mb x}$}\}
$$
is finite.  In particular, quotient of $\mc P_*(\Gamma,|C|)$ by
congruence relation
$\mc P_*(\Gamma,|C|)$ 
contains uncountably many curved foldings which 
are not congruent to each other.  
\item[(c)] Suppose that
$C$ and $\mu$ have no symmetries $($cf. Definition \ref{ft-sym2}$)$.
Then for each $i\in \{1,2,3,4\}$ and $\mb x\in C$, the set 
$\Lambda^i_{\mb x}$ consists of a single element,
that is, any two curved foldings in 
$\mc P_*(\Gamma,|C|)$ are mutually non-congruent. 
\end{enumerate}}

In Theorem C, we do not need to assume that $\Gamma$ is a
closed curve in $\R^2$. However,
the most interesting case is that $C$ and $\Gamma$ are both closed: 
At the end of Section~5, we  give a concretely
example of a $\mc P_*(\Gamma,|C|)$ containing 
uncountably many congruence classes.

Theorems A, B and C can be considered as the analogues 
for cuspidal edges along a space curve
given in \cite[Theorem III and Theorem IV]{HNSUY} 
and \cite[Theorem 1.8]{HNSUY2}, respectively.
However, even if $\Gamma$ and $C$ admit 
real analytic parametrizations,  
Theorems A, B and C do not directly follow 
from the corresponding assertions for cuspidal edges.

\medskip
\section{Preliminaries}
We let $C$ be the image of an embedded space curve whose 
length is $l$. When $C$ is a non-closed curve, 
it has a parametrization $\mb c:\mathbb I_l\to \R^3$
with arc-length parameter, where $\mathbb I_l:=[-l/2,l/2]$.
On the other hand, if $C$ is closed (i.e.~a knot), 
then it can be parametrized by a curve
$\mb c:\mathbb T^1_l\to \R^3$ with arc-length parameter,
where $\mathbb T^1_l:=\R/l\Z$.
Since we treat the bounded closed interval $\mathbb I_l$ 
and the one dimensional torus $\mathbb T^1_l$
at the same time, we set
$$
J:=\mathbb I_l \quad \text{or}\quad \mathbb T^1_l.
$$
As explained in the introduction, $C$ 
has the orientation induced by the parametrization $\mb c$. 
We let $-C$ be the curve $C$ whose orientation 
is reversed, and $|C|$ 
denotes the curve $C$ ignoring its orientation.
We consider special strips along $C$, which are
\lq\lq developable strips'' along $C$. Let $J_0$ be a set which 
is homeomorphic to $J$.

\begin{Def}
A {\it developable strip} along $C$ is a $C^\infty$-embedding 
$f:U\to \R^3$ defined on a tubular neighborhood $U$ 
of $J_0\times \{0\}$ in $J_0\times \R$ such that 
\begin{itemize}
\item $J_0 \ni u\mapsto f(u,0)\in \R^3$ parametrizes $C$,
\item there exists a unit vector field $\xi_f(u)$  
of $f$ along $C$ $($called a {\it ruling vector field}$)$
such that $f$ can be expressed as
$$
f(u,v)=f(u,0)+v \xi_f(u)\qquad ((u,v)\in U),\,\, \text{and}
$$
\item the Gaussian curvature of $f$ vanishes on $U$ identically.
\end{itemize}
\end{Def}

The developable strip $f$ represents a map germ along $C$.
We identify this induced map germ with $f$ itself if it creates no confusion.
Hereafter, we assume that the curvature function of $C$
never vanishes, and we denote by $\mb e(u)$, 
$\mb n(u)$ and $\mb b(u)$ the unit tangent, unit principal normal and unit bi-normal vector fields associated with the 
parametrization 
\begin{equation}\label{eq:Cf}
\mb c_f:J_0\ni u\mapsto f(u,0)\in \R^3
\end{equation}
of $C$, respectively. With these notations,  we can express $\xi_f$ as
\begin{equation}\label{eq:f-1}
\xi_f(u)=\cos \beta_f(u)\mb e(u)+ \sin \beta_f(u)
\Big(
\cos \alpha_f(u)\mb n(u)+\sin \alpha_f(u)\mb b(u)
\Big).
\end{equation}
This $\alpha_f:J_0\to \R$ is called the 
{\it first angular function},
and $\beta_f:J_0\to \R$ is 
called the {\it second angular function} of $f$.

In this paper, we consider the developable strips satisfying
\begin{equation}\label{eq:alpha_f0}
0<|\cos \alpha_f (u)|<1 \qquad (u\in J_0).
\end{equation}
We denote by $\mc D(C)$ the set of developable strip germs along $C$ satisfying \eqref{eq:alpha_f0}. Moreover, $f\in \mc D(C)$ is said to be
 {\it admissible} if it satisfies the following stronger 
condition
\begin{equation}\label{eq:alpha_f2}
0<|\kappa_f(u)
\cos \alpha_f(u)|<\min_{w\in J_0} \kappa_f(w)
\qquad (u\in J_0),
\end{equation}
which corresponds to the condition (iv${}'$) 
in the introduction,
where $\kappa_f$ is the curvature function of 
$\mb c_f$ (cf. \eqref{eq:Cf}).
Let ${\mc D}_*(C)$ be the set of admissible 
developable strip germs along $C$. By definition,
$$
\mc D_{*}(C)\subset \mc D(C)
$$
holds. We set
$$
{\mc D}(|C|):={\mc D}(C)\cup {\mc D}(-C),\qquad
{\mc D}_*(|C|):={\mc D}_*(C)\cup {\mc D}_*(-C).
$$
By \eqref{eq:alpha_f0},
we can choose the first angular function $\alpha_f$ so that 
\begin{equation}\label{eq:alpha_f}
0<|\alpha_f(u)|<\frac{\pi}{2} \qquad (u\in J_0).
\end{equation}
The Gaussian curvature of $f$ vanishes identically if and only if
$$
\det(f_u(u,0),\xi_f(u),\xi'_f(u))=0 \qquad
\left(\xi'_f(u):=\frac{d \xi_f(u)}{du}\right),
$$
which is equivalent to the formula
\begin{equation}\label{eq:beta_f}
\cot \beta_f(u)=\frac{\alpha'_f(u)+|\mb c'_f(u)|
\tau_f(u)}{|\mb c'_f(u)|\kappa_f(u)\sin \alpha_f(u)},
\end{equation}
where $\tau_f(u)$ is the torsion function of $\mb c_f(u)$.
In particular, we may assume that \begin{equation}\label{eq:beta_f2}
0<\beta_f(u)<\pi \qquad (u\in J_0).
\end{equation}
\underline{Throughout this paper, we assume \eqref{eq:alpha_f}
and \eqref{eq:beta_f2} for $f\in \mc D(C)$.}
In particular,  $\xi_f(u)$ satisfies
\begin{equation}\label{eq:xi-n}
\xi_f(u)\cdot \mb n(u)>0 \qquad (u\in J_0).
\end{equation}
Such a $\xi_f$ is called the {\it normalized ruling vector field} of $f$.
Then
$$
\mb N_f(u):=\cos \alpha_f(u)\mb n(u)+\sin \alpha_f(u)\mb b(u)
$$
gives the unit co-normal vector field of $f$ along $C$
satisfying $\mb N_f\cdot \xi_f>0$. We set
\begin{equation}\label{eq:mu-f}
\mu_f(u):=\frac{\mb c''_f(u)\cdot \mb N_f(u)}{|\mb c'_f(u)|^2}
=\kappa_f(u)\cos \alpha_f(u),
\end{equation}
which is a positive-valued function giving the geodesic curvature
of $C$ as a curve on the surface $f$. 
We call $\mu_f$ the {\it geodesic curvature function} 
of $f$ along $C$. Since $\mu_f(u)>0$, \eqref{eq:alpha_f} and
\eqref{eq:alpha_f2} reduce to the conditions
\begin{equation}\label{eq:alpha_f20}
(0<)\mu_f(u)< \kappa_f(u)\qquad (u\in J_0),
\end{equation}
and  
\begin{equation}\label{eq:alpha_f2B}
(0<)\mu_f(u)<\min_{w\in J_0} \kappa_f(w)\qquad (u\in J_0),
\end{equation}
respectively. Let $l$ be the total arc-length of $C$. 
Then the parameter $u$ of $\mb c_f$ can be expressed as 
$u=u(s)$ ($s\in J$) so that $\mb c(s):=\mb c_f(u(s))$
 ($s\in J$)
gives the arc-length parametrization of $C$.
In this situation, the function defined by
$$
\hat \mu_f(s):=\mu_f(t(s))\qquad (s\in J)
$$
is called the {\it normalized geodesic curvature function} of $f$.
We now fix a point  
\begin{equation}\label{eq:Bpt}
{\mb x}_0 \in C,
\end{equation}
which is the midpoint of $C$ if
$J=\mathbb I_l$ and is
an arbitrarily chosen point if $J=\mathbb T^1_l$.

\begin{Def}\label{def:NM}
Let $f\in \mc D(C)$. 
We call $f(s,v)$ a {\it normal form} of a developable strip
if it is defined on
a tubular neighborhood of $J\times \{0\}$
in $\R^2$
and 
$$
J\ni s\mapsto f(s,0)\in \R^3
$$
gives an arc-length parametrization of $C$.
Since we will use $s$ to denote the arc-length parameter of $C$,
we use the parametrization $f(s,v)$ when $f$ is a normal form.
(We shall denote such developable strips using capital letters
to emphasize that they are written in normal forms.)
\end{Def}

Suppose that $f(s,v)$ is a normal form.
As seen in \cite{Ori}, the restriction $H(s)$
of the mean curvature function of $f(s,v)$ to 
the curve $\mb c(s)$  ($s\in J$) satisfies
\begin{equation}\label{eq:H}
|H(s)|=\frac{\kappa(s)^2\sin^2 \alpha_f(s)
+(\alpha'_f(s)+\tau(s))^2}{2\kappa(s)|\sin \alpha_f(s)|},
\end{equation}
where $\kappa(s)$ and $\tau(s)$ are the curvature 
and torsion function of $\mb c(s)$.
By \eqref{eq:alpha_f}, the right-hand side of \eqref{eq:H}
never vanishes. In particular, $f$ has no umbilics and the ruling 
direction $\xi_f(s):=f_v(s,0)$ along $C$
points in the (uniquely determined) asymptotic direction of $f$.
Hence, the germ of the normal form of $f$ is determined by 
the base point $\mb x_0\in C$ and the first angular function $\alpha_f$.

From now on, we again assume that $u$ is a general parameter, that is,
it may not be an arc-length parameter of $C$, in general.
Let $J_0$ be a set which is homeomorphic to $J$.
We let 
\begin{equation}\label{def:1}
C^\infty_{\pi/2}(J_0)
\end{equation}
be the set of $C^\infty$-functions defined on $J_0$ 
whose images lie in the set 
$
(-\pi/2,\pi/2)\setminus\{0\}.
$
Then the first angular function $\alpha_f(u)$ of $f(u,v)$
belongs to this class $C^\infty_{\pi/2}(J_0)$.

\begin{Remark}\label{rmk:c-sharp}
When $J_0=[b,c]$ ($b<c$), we set 
$\mb c^\sharp(u):=\mb c_f(b+c-u)$,
which has the same image as $\mb c_f(u)$ (cf. \eqref{eq:Cf})
but has the opposite orientation.
Then
\begin{equation}\label{eq:s2}
\mb e^\sharp(u):=-\mb e(b+c-u),\quad
\mb n^\sharp(u):=\mb n(b+c-u),\quad
\mb b^\sharp(u):=-\mb b(b+c-u)
\end{equation}
are the unit velocity vector, the unit principal normal vector and
the unit bi-normal vector of the curve $\mb c^\sharp(u)$, respectively.
If we denote by $\tau_f(u)$ the torsion function of $\mb c_f(u)$,
then  
\begin{equation}\label{eq:s1}
\kappa^\sharp(u):=\kappa_f(b+c-u),
\qquad \tau^\sharp(u):=\tau_f(b+c-u)
\end{equation}
coincide with the curvature and torsion functions of $\mb c^\sharp(u)$, 
respectively. For each $f\in \mc D(C)$, we set 
\begin{equation}
f^\sharp(u,v):=f(b+c-u,v),
\end{equation}
and call this the {\it reverse} of $f$.
By definition, $f^\sharp$ has the same image as $f$, and
the involution
$
{\mc D}(|C|)\ni f \mapsto f^\sharp\in {\mc D}(|C|)
$
is canonically induced. By \eqref{eq:s2}, 
$
\xi^\sharp_f(u):=\xi_f(b+c-u)
$
gives the normalized ruling vector field of $f^\sharp$
if this is so of  $\xi_f(u)$ for $f$.
Then the first and second angular functions 
$\alpha^\sharp,\,\, \beta^\sharp$
of $f^\sharp$ satisfy
\begin{equation}\label{eq:alpha-s}
\alpha^\sharp(u)=-\alpha_f(b+c-u),\qquad 
\beta^\sharp(u)=\pi-\beta_f(b+c-u),
\end{equation}
respectively.
\end{Remark}

\begin{Def}\label{def:C}
Let $J_i$ ($i=1,2$) be two sets which are homeomorphic to $J$, 
and let $f_i:J_i\times (-\epsilon_i,\epsilon_i)
\to \R^3$ ($i=1,2$)
be two developable strips along $C$, where
$\epsilon_i>0$. 
Then $f_2$ is said to be {\it image equivalent}
(resp.  {\it right equivalent})
to $f_1$ if there exists a positive 
number $\delta(<\min(\epsilon_1,\epsilon_2))$
such that $f_1(J_1 \times (-\delta,\delta))$
coincides with $f_2(J_2 \times (-\delta,\delta))$
(resp. $f_2$ coincides with $f_1\circ \phi$ 
for a diffeomorphism
$
\phi:J_1\times (-\delta,\delta)
\to J_2\times (-\delta,\delta)
$).
\end{Def}

Recall that $s\mapsto \mb c(s)$ ($s\in J$)
gives an arc-length parametrization of $C$ such that $\mb c(0)={\mb x}_0$. The following assertion holds.

\begin{Proposition}\label{prop:U}
Let $F,\, G\in \mc D(|C|)$ be normal forms\footnote{
From now on, we use the capital letters
$F,\, G$ to express the normal forms of
developable strips, and use lower-case
letters $f,\, g$ to express
general developable strips.} $($cf. Definition \ref{def:NM}$)$
satisfying $F(0,0)=G(0,0)={\mb x}_0$.
Then the following two assertions are equivalent:
\begin{enumerate}
\item $F=G$ or $F=G^\sharp$,
\item $F$ is image equivalent to $G$.
\end{enumerate}
In particular, for each $\alpha\in C^\infty_{\pi/2}(J)$, 
there exists a unique normal form 
\begin{equation}\label{eq:FA}
F^\alpha\in \mc D(C)
\end{equation}
satisfying $F^\alpha(0,0)={\mb x}_0$ 
$($cf. \eqref{eq:Bpt}$)$
whose first 
angular function is $\alpha$.
\end{Proposition}

\begin{proof} Replacing $G$ by $G^\sharp$, 
we may assume that $F,\, G\in \mc D(C)$. (In fact, 
$G^\sharp$ is a normal form if so is $G$.)
It is obvious that (1) implies (2). 
So it is sufficient to show the converse. 
Since each normal form of a developable strip
is determined by its base point, 
the arc-length parametrization of $C$
and the first angular function, (2) implies (1).
\end{proof}

We also prove the following assertion:

\begin{Proposition}\label{prop:fF}
For each $f\in \mc D(C)$, there exists a unique normal 
form $F\in \mc D(C)$
such that
\begin{enumerate}
\item $F(0,0)=\mb x_0$ $($if $J=\mathbb I_l$, this
holds automatically$)$, and
\item $F$ is right equivalent to $f$,
\end{enumerate}
where $\mb x_0\in C$  is the base point given in
\eqref{eq:Bpt}.
Moreover, 
$\hat \mu_f=\hat \mu_F=\mu_F$
hold.
\end{Proposition}

We call this $F$ the {\it normal form associated with $f$ of base point $\mb x_0$}.

\begin{proof}
Applying \cite[Lemma B.5.3]{UY-book},
we can take a curvature line coordinate system $(s,v)$ of $f$
such that $s \mapsto f(s,0)$ parametrizes $C$ and
$df(\partial/\partial v)$ points in the ruling direction.
In this situation, we may assume that 
$f(0,0)=\mb x_0$ and $s$ is the
arc-length parameter of $C$. We then adjust $v$ 
so that $|f_v(s,0)|=1$
for $s\in J$.
Since the image of each $v$-curve gives a straight line,
this parametrization $f(s,v)$ gives the normal form
associated to $f$.
\end{proof}

We next prepare the following two lemmas, 
which will be applied
in the later discussions.

\begin{Lemma}\label{lem:Tv}
Let $F_i\in \mc D(C)$ $(i=1,2)$
be normal forms
satisfying $F(0,0)=G(0,0)=\mb x_0$,
and let $\alpha_i\in C^\infty_{\pi/2}(J)$  
be its first angular functions.
If $\alpha_2-\alpha_1$ does not have any zeros on $J$,
then the ruling direction of $F_1$ $($cf. \eqref{eq:FA}$)$
is 
linearly independent of
that of $F_2$ at each point of $C$.
\end{Lemma}

\begin{proof}
We let $\xi_i(s)$ ($i=1,2$) be the normalized unit 
ruling vector fields
 associated with $F^{\alpha_i}(s,v)$. 
Since $\alpha_2(s)\ne \alpha_1(s)$, two vectors 
$$
\cos \alpha_1(s) \mb n(s)+\sin \alpha_1(s) \mb b(s),\qquad
\cos \alpha_2(s) \mb n(s)+\sin \alpha_2(s) \mb b(s)
$$
in $\R^3$ are linearly independent for each $s\in J$. 
So we obtain the assertion. 
\end{proof}

\begin{Proposition}\label{L0}
Let $f\in \mc D(C)$.
If $T$ is a symmetry of $C$ $($cf. Definition \ref{def:S}$)$, then $T\circ f$ belongs to $\mc D(|C|)$.
Moreover, if $f$ is a normal form,
then so is $T\circ f$, and $dT(\xi_f)$ gives the
normalized ruling vector field of $T\circ f$.
Furthermore, the normalized geodesic curvature function
of $T\circ f$ coincides with that of $f$.
\end{Proposition}

\begin{proof}
We denote by $F$ the normal form associated with $f$ 
of base point $\mb x_0$
(cf. Proposition~\ref{prop:fF}). 
It is sufficient to show the assertion holds for $F$.
Since the property that the geodesic curvature 
has no zeros is preserved by isometries of $\R^3$,
the first assertion is obtained.
Since $s\mapsto T\circ F(s,v)$ gives an 
arc-length parametrization of $|C|$, $T\circ F(s,v)$
is a normal form.  Since the principal normal vector field
is common in $C$ and $-C$ (cf. \eqref{prop:03a}), 
it can be easily seen that
$dT(\mb n(s))$ gives the principal 
normal vector field of $T\circ \mb c(s)$, and
\eqref{eq:xi-n} yields that
\begin{equation}\label{eq:xi-n2}
dT(\xi_F(s))\cdot dT(\mb n(s))
=\xi_F(s)\cdot \mb n(s)>0 
\end{equation}
along $C$, which implies the second assertion.
Since geodesic curvatures on surfaces 
are geometric invariants up to $\pm$-ambiguities, 
the geodesic curvature of $T\circ F$ 
coincides with $\sigma \mu_F$, where $\sigma\in \{1,-1\}$.
Moreover, \eqref{eq:xi-n2}
implies that $\sigma=1$, proving the last assertion.
\end{proof}

%\medskip
We set
\begin{equation}\label{eq:D}
\Omega^+_{\epsilon}:=J\times (0,\epsilon),
\quad \Omega^-_{\epsilon}=J\times (-\epsilon,0), 
\quad \Omega_{\epsilon}:=J\times (-\epsilon,\epsilon).
\end{equation}

\begin{Proposition}\label{fact:1} 
For $f,\, g\in \mc D(|C|)$,
the following five conditions are equivalent:
\begin{enumerate}
\item $f$ is right equivalent to $g$
as map germs,
\item $F=G$ or $F=G^\sharp$,
where $F$ and $G$ are the normal forms associated with $f$ and $g$
satisfying $F(0,0)=G(0,0)=\mb x_0$,
respectively,
\item $F(\Omega^+_{\epsilon})=G(\Omega^+_{\epsilon})$
for sufficiently small $\epsilon(>0)$,
\item $F(\Omega_{\epsilon})=G(\Omega_{\epsilon})$
for sufficiently small $\epsilon(>0)$,
\item  $f$ is image equivalent to $g$ as map germs. 
\end{enumerate}
\end{Proposition}

\begin{proof}
By Proposition \ref{prop:fF},
 (1) implies (2), because $F(0,0)=G(0,0)=\mb x_0$. 
Obviously (2) implies (3).
Moreover, (3) implies (4),
because $F(s,v)$ and $G(s,v)$  are ruled strips 
and are real analytic with respect to the
parameter $v$.
On the other hand, it is obvious that (5) is
equivalent to (4). 
If (4) holds, then  
Proposition \ref{prop:U} yields
that $F=G$ or $F=G^\sharp$.
In particular, (1) is obtained.
\end{proof}

We next prove the following assertion,
which is a refinement of \cite[Lemma 1.2]{Ori}.

\begin{Proposition}\label{lem:key}
Let $F_i\in \mc D(C)$ $(i=1,2)$
be developable strips written in normal forms
satisfying $F(0,0)=G(0,0)=\mb x_0$.
If their normalized ruling vector fields
are linearly independent at each point of
$C$, then $F_1(\Omega_\epsilon)\cap F_2(\Omega_\epsilon)$ 
coincides with $C$
for sufficiently small $\epsilon(>0)$.
\end{Proposition}

\begin{proof}
Without loss of generality,
we may assume that
$F_1$ and $F_2$ are defined on a 
tubular neighborhood of $J\times \{0\}$ in $J\times \R$
and $F_1(s,0)=F_2(s,0)$ holds for $s\in J$.
Suppose that the assertion fails.
Then there exist 
\begin{itemize}
\item two sequences $\{s_n\}_{n=1}^\infty$ and $\{t_n\}_{n=1}^\infty$ on 
$J$, and
\item two sequences $\{u_n\}_{n=1}^\infty$ and $\{v_n\}_{n=1}^\infty$ 
on $(-1/n,1/n)$
\end{itemize}
such that
\begin{equation}\label{eq:star0}
F_1(s_n,u_n)=F_2(t_n,v_n), \qquad (s_n,u_n)\ne (t_n,v_n).
\end{equation}
Here, $s_n\ne t_n$ holds. (In fact,
if not, then $F_1(s_n,u_n)=F_2(t_n,v_n)$ implies
$u_n \xi_1(s_n)=v_n \xi_2(s_n)$,
where $\xi_i$ ($i=1,2$) is
the normalized ruling vector field of $F_i$.
However, since $\{\xi_1(s_n),\xi_2(s_n)\}$
is linearly independent, the fact
$u_n \xi_1(s_n)=v_n \xi_2(s_n)$ implies $u_n=v_n=0$, which contradicts
the fact $(s_n,u_n)\ne (t_n,v_n)$.)

Since $J$ is compact, we may assume that
the limits
$\dy\lim_{n\to \infty}s_n=s_\infty$ and
$\dy\lim_{n\to \infty}t_n=t_\infty$
exist and $s_\infty,t_\infty\in J$.
Then by \eqref{eq:star0}, we have
$
F_1(s_\infty,0)=F_2(t_\infty,0).
$
Since $C$ has no self-intersections, we have
$
s_\infty=t_\infty
$.
By \eqref{eq:star0}, we can write
\begin{equation}\label{eq:DD}
\frac{\mb c(s_n)-\mb c(t_n)}{s_n-t_n}
=p_n\xi_1(s_n)-q_n \xi_2(t_n)
\qquad \Big(p_n:=\frac{-u_n}{s_n-t_n},\quad q_n:=\frac{-v_n}{s_n-t_n}\Big).
\end{equation}
If $n\to \infty$, then the left-hand side of
\eqref{eq:DD} converges to the vector
$\mb c'(s_\infty)(=\mb e(s_\infty))$.
Thus, we can conclude that the limits
$\dy\lim_{n\to \infty}p_n=p_\infty$ and 
$\dy\lim_{n\to \infty}q_n=q_\infty$
exist such that
\begin{equation}\label{eq:UV}
\mb e(s_\infty)=p_\infty\xi_1(s_\infty)
-q_\infty \xi_2(s_\infty).
\end{equation}
In particular, we have
\begin{equation}\label{eq:UV2}
(p_\infty,q_\infty)\ne (0,0).
\end{equation}
We set
$$
\mb e_\infty:=\mb e(s_\infty),\quad \mb n_\infty:=\mb n(s_\infty), \quad 
\mb b_\infty:=\mb b(s_\infty)
$$
and
$$
A_1:=\alpha_1(s_\infty),\quad
A_2:=\alpha_2(s_\infty),\quad
B_1:=\beta_1(s_\infty),\quad
B_2:=\beta_2(s_\infty),
$$
where $\alpha_i$ and $\beta_i$ ($i=1,2$)
are the first and second angular
functions of $F_i$, respectively.
Multiplying $\mb n_\infty$ and $\mb b_\infty$ to \eqref{eq:UV}
via inner products, we have
$$
p_\infty\cos A_1 -q_\infty\cos A_2 =0, \quad
p_\infty\sin A_1 -q_\infty\sin A_2 =0.
$$
By \eqref{eq:UV2}, we have
\begin{equation}\label{eq:A1A2}
0=\cos A_1 \sin A_2-\cos A_2\sin A_1=\sin(A_1-A_2).
\end{equation}
On the other hand, since $F_1,F_2\in \mc D(C)$
(cf. \eqref{eq:alpha_f0})
 and $\xi_1(s_\infty)$ is linearly
independent of $\xi_2(s_\infty)$,
we have
$$
0<|A_1-A_2|<|A_1|+|A_2|<\pi,
$$
which contradicts \eqref{eq:A1A2}.
We remark that a similar argument 
for developable surfaces is used 
in \cite[Section 5]{MU}.
\end{proof}

Using the same technique, we can  prove the
following assertion:

\begin{Proposition}\label{lem:pm}
Let $F,G\in \mc D(|C|)$ 
be developable strips written in normal forms.
Then $F(\Omega^+_\epsilon)$
does not meet
$G(\Omega^-_\epsilon)$ 
for sufficiently small $\epsilon(>0)$.
\end{Proposition}

\begin{proof}
Replacing $G$ by $G^\sharp$, 
we may assume that $F,\, G\in \mc D(C)$.
Moreover, replacing $G(s,v)$  by
$G(s+b,v)$ for a suitable $b\in [0,l)$
if necessary,
we may assume that
$F,G$  are defined on a 
tubular neighborhood of $J\times \{0\}$ in $J\times \R$
and $F(s,0)=G(s,0)$ holds for $s\in J$.
We denote by $\xi_F$ and $\xi_G$
the normalized vector fields of $F$ and $G$, respectively.
Then we may assume that
$F$ and $G$ are defined on a tubular
neighborhood of $J\times \{0\}$ in $J\times \R$.
Suppose that the assertion fails.
Then there exist 
\begin{itemize}
\item two sequences $\{s_n\}_{n=1}^\infty$ and $\{t_n\}_{n=1}^\infty$ on 
$J$, and
\item two sequences $\{u_n\}_{n=1}^\infty$ and $\{v_n\}_{n=1}^\infty$ 
on $(0,1/n)$ and $(-1/n,0)$, respectively,
\end{itemize}
such that $F(s_n,u_n)=G(t_n,v_n)$
and $(s_n,u_n)\ne (t_n,v_n)$.
In this situation, we can show that $s_n\ne t_n$.
(In fact, if $s_n=t_n$, then
$F(s_n,u_n)=G(t_n,v_n)$ implies
$
u_n \xi_F(s_n)=v_n \xi_G(s_n)
$
and so we have
$$
u_n \xi_F(s_n)\cdot \mb n(s_n)=v_n \xi_G(s_n)\cdot \mb n(s_n).
$$
Here, 
$\xi_F(s_n) \cdot \mb n(s_n)$ and
$\xi_G(s_n)\cdot \mb n(s_n)$
are positive (cf. \eqref{eq:xi-n}).
So this contradicts the facts $u_n\in (0,1/n)$
and $v_n\in (-1/n,0)$.)

Since $J$ is compact, we may assume that
the limits $\dy\lim_{n\to \infty}s_n=s_\infty$
and $\dy\lim_{n\to \infty}t_n=t_\infty$
exist and $s_\infty,t_\infty\in J$.
Since $C$ has no self-intersections, we have
$
s_\infty=t_\infty.
$
By \eqref{eq:star0}, we can write
\begin{equation}\label{eq:DD2}
\frac{\mb c(s_n)-\mb c(t_n)}{s_n-t_n}
=p_n\xi_F(s_n)-q_n \xi_G(t_n)
\qquad \Big(p_n:=\frac{-u_n}{s_n-t_n},\quad q_n:=\frac{-v_n}{s_n-t_n}\Big).
\end{equation}
If $n\to \infty$, then the left-hand side of
\eqref{eq:DD} converges to the vector
$\mb c'(s_\infty)(=\mb e(s_\infty))$.
Thus, we can conclude that the limits
$\dy\lim_{n\to \infty}p_n=p_\infty$ and
$\dy\lim_{n\to \infty}q_n=q_\infty$
exist such that
\begin{equation}
\label{eq:UV2b}
\mb e(s_\infty)=p_\infty\xi_F(s_\infty)
-q_\infty \xi_G({s_\infty}).
\end{equation}
Since the left hand side does not vanish,
we have
\begin{equation}
\label{eq:UV2a}
(p_\infty,q_\infty)\ne (0,0).
\end{equation}
Taking the inner product of $\mb n(s_\infty)$ to
\eqref{eq:UV2b}, we have
\begin{equation}\label{eq:pq22}
0=p_\infty \, \xi_F(s_\infty)\cdot \mb n(s_\infty)
-q_\infty \,\xi_G(s_\infty)\cdot \mb n(s_\infty).
\end{equation}
Since $u_n>0$ and $v_n<0$, we have
$p_\infty q_\infty\le 0$.
Since 
$\xi_F(s_\infty) \cdot \mb n(s_\infty)$ and
$\xi_G(s_\infty)\cdot \mb n(s_\infty)$
are positive,
\eqref{eq:UV2a} implies that
the right hand side of
\eqref{eq:pq22} does not vanish, a contradiction.
\end{proof}

\section{The dual developable strips and curved foldings}

In this section, we explain
curved foldings using pairs of developable strips:

\begin{Definition}\label{ft-sym}
Let $J_i$ ($i=1,2$) be bounded closed intervals of $\R$.
Two functions $\mu_i:J_i\to \R$ ($i=1,2$) are said to be
{\it equi-affine equivalent} 
if there exists a diffeomorphism $\phi:J_1\to J_2$ 
of the form 
$$
\phi(u)=\sigma u+d \qquad (\sigma\in \{1,-1\},\,\, d\in \R)
$$
such that
$
\mu_2\circ \phi=\mu_1.
$
(By definition, if $\mu_2$ is equi-affine equivalent 
to $\mu_1$,
then the length of the interval $J_2$ must be 
equal to that of $J_1$.)

In the case of $J_1=J_2=[b,c]$ ($b<c$)
and $(\mu:=)\mu_1=\mu_2$, the map 
$\phi:J_1\to J_1$ is called a {\it symmetry} 
of $\mu$ if $\phi$ is not the identity. 
(There is at most one possibility for
 such a $\phi$, which must
have the expression $\phi(u)=b+c-u$.
So, if  such a $\phi$ exists, $\mu$ satisfies
$\mu(u)=\mu(b+c-u)$ on $[b,c]$.)
\end{Definition}

\begin{Lemma}\label{Aeq1}
For each $i\in \{1,2\}$, 
let $\gamma_i:\mathbb I_l\to \R^2$
$(l>0)$
be a regular curve
without self-intersections
parametrized by
arc-length. 
Suppose that the curvature functions of $\gamma_1$
and $\gamma_2$ are positive-valued.
Then the following two assertions
are equivalent:
\begin{enumerate}
\item the curvature function of $\gamma_2$
is equi-affine equivalent to that of $\gamma_1$,
\item there exists an isometry $T$
of $\R^2$ such that
$T(\gamma_1(\mathbb I_l))$ coincides with $\gamma_2(\mathbb I_l)$.
\end{enumerate}
\end{Lemma}

\begin{proof}
We suppose (1).
We set
$\gamma_1^\sharp(s):=S\circ \gamma_1(-s)$
($s\in \mathbb I_l$),
where $S$ is the reflection with respect to a straight
line in $\R^2$. 
We denote by $\mu_i$  ($i=1,2$)
the  curvature function of $\gamma_i$.
If $\mu_2$ is equi-affine equivalent to $\mu_1$,
then $\mu_2$ coincides with 
the curvature function of $\gamma_1$ or $\gamma_1^\sharp$.
By the fundamental theorem of curves 
in the Euclidean plane (cf. \cite[Chapter 2]{UY-book}),
there exists an isometry $T$ in $\R^2$
such that 
$T\circ \gamma_1=\gamma_2$
or $T\circ \gamma_1^\sharp=\gamma_2$,
which implies (2).

Conversely, we suppose (2).
Since 
$\gamma_1$ and $\gamma_2$
have no self-intersections,
such an isometry $T$ is uniquely determined.
Since 
$\gamma_1$ and $\gamma_2$
are parametrized by arc-length,
either $T\circ \gamma_1(s)=\gamma_2(s)$
or $T\circ \gamma_1(s)=\gamma_2(-s)$ holds on $\mathbb I_l$.
Since the curvature functions of $\gamma_1$
and $\gamma_2$ are positive-valued,
$$
T\circ \gamma_1(s)=\gamma_2(s) \qquad 
(\text{resp. } T\circ \gamma_1(s)=\gamma_2(-s))
$$
holds on $\mathbb I_l$
if $T$ is an orientation preserving
(resp. reversing) isometry of $\R^3$, which implies
$\mu_1(s)=\mu_2(s)$ 
(resp. $\mu_1(s)=\mu_2(-s)$) for $s\in \mathbb I_l$.
So (1) holds.
\end{proof}

Let $\mu:\mathbb T^1_a\to \R$ be a $C^\infty$-function on 
a one dimensional torus $\mathbb T^1_a:=\R/a\Z$ ($a>0$).
Then an $a$-periodic function 
$\tilde \mu:\R\to \R$ defined by
\begin{equation}\label{eq:mu111}
\tilde \mu:=\mu\circ \pi
\end{equation}
is called the {\it lift} of the function $\mu$,
where $\pi:\R\to \mathbb T^1_a$ is the canonical
projection.

\begin{Definition}\label{ft-sym2}
We set $J_i:=\R/a_i\Z$ ($a_i>0,\,\, i=1,2$).
Two functions $\mu_i:J_i\to \R$ ($i=1,2$) are said to be
{\it equi-affine equivalent} 
if there exists a diffeomorphism $\phi:\R\to \R$ 
of the form 
$$
\phi(u)=\sigma u+d \qquad (\sigma\in \{1,-1\},\,\, d\in \R)
$$
such that
$
\tilde \mu_2\circ \phi=\tilde \mu_1,
$
where $\tilde \mu_i$ $(i=1,2)$
are the lifts of the functions $\mu_i$.
In the case of $(\mu:=)\mu_1=\mu_2$ and $J_1=J_2$, such a map
$\phi:\R\to \R$ is called a {\it symmetry} 
of $\mu$ if $\phi$ is non-trivial, that is,
either $\sigma=-1$ or $d\not \in a\Z$ holds, 
where $a:=a_1(=a_2)$.
If $\mu$ is a non-constant function, then
$\phi$ can be a candidate for symmetries
of $\mu$ only when $d$ belongs
to $(a\Q)\setminus (a\Z)$, 
where $\Q$ is the set of rational numbers.
\end{Definition}

\begin{Lemma}\label{Aeq2}
Let $J_i$ $(i=1,2)$ be two bounded closed intervals, and
let $\gamma_i:J_i\to \R^2$  be plane curves
of length $l$ parametrized by
arc-length. 
Suppose that each curvature function of
$\gamma_i$ $(i=1,2)$ has a $C^\infty$-extension 
$\tilde \mu_i:\R\to \R$ which is 
the lift of an $l$-periodic function 
$\mu_i:\mathbb T^1_l \to \R$.
Then the following two assertions
are equivalent:
\begin{enumerate}
\item the function $\mu_2$
is equi-affine equivalent to $\mu_1$,
\item there exist a plane curve $\tilde \gamma:\R\to \R^2$
and an orientation preserving isometry $T$ of $\R^3$
such that $\gamma_1(J_1)$ and
$T\circ \gamma_2(J_2)$ 
are subarcs of $\tilde \gamma(\R)$.
\end{enumerate}
\end{Lemma}

\begin{proof}
We suppose (1).
Since each curvature function of
$\gamma_i$ ($i=1,2$)
can be extended as an
$l$-periodic $C^\infty$-function $\tilde \mu_i$ on $\R$,
the curve $\gamma_i$ is extended as a regular curve
$\tilde \gamma_i:\R\to \R^2$
whose curvature  function is $\tilde \mu_i$.
If (1) holds, then
there exist $\sigma\in \{1,-1\}$ and $d\in [0,l)$
such that
$$
\tilde \mu_1(s)=\tilde \mu_2(\sigma s+d)\qquad (s\in \R).
$$
By the fundamental theorem of plane curves,
$\tilde \gamma_2(s)$ coincides with
$T\circ \tilde \gamma_1(\sigma s+d)$,
where $T$ is an orientation preserving isometry of $\R^2$.
By setting $\tilde \gamma:=\tilde \gamma_1$,
(2) is obtained.
On the other hand, the converse assertion can be  proved easily.
\end{proof}

We now define
the \lq\lq geodesic equivalence relation'' on $\mc D(|C|)$
as follows:

\begin{Definition}\label{Def:Eqe1}
Let $C$ be a non-closed space curve (i.e. $C:=\mb c (\mathbb I_l)$) 
or a closed curve (i.e. $C:=\mb c (\mathbb T^1_l)$) of total length $l$
embedded in $\R^3$.
Two developable strip germs $f,\, g\in {\mc D(|C|)}$ are
said to be {\it geodesically equivalent} if the normalized
geodesic curvature function $\hat \mu_{f}:J\to \R$
is equi-affine equivalent to $\hat \mu_{g}:J\to \R$,
where $J=\mathbb I_l$ or $J=\mathbb T^1_l$.
\end{Definition}

The following assertion holds:

\begin{Prop}\label{prop:2eq}
Let $f,g\in \mc D(|C|)$.
If $f$ and $g$ are right equivalent, then they are 
geodesically equivalent.
\end{Prop}

\begin{proof}
By replacing $g$ by $g^\sharp$,
we may assume that $f,g\in \mc D(C)$. 
We denote by $F$ and $G$ the normal forms associated with
$f$ and $g$ satisfying $F(0,0)=G(0,0)=\mb x_0$, respectively. 
Then by Corollary 1.10, $F$ coincides with $G$, which implies
$
\mu_f=\mu_F=\mu_G=\mu_g.
$
\end{proof}

Later, we will see that the geodesic equivalence relation
is useful for constructing curved foldings with a given crease and crease pattern.
Based on this, we give the following:

\begin{Def}\label{def:IS}
For $f\in {\mc D(|C|)}$,
a developable strip germ $g\in {\mc D(|C|)}$ is called an {\it isomer}
of $f$ if 
\begin{enumerate}
\item $g$ is geodesically equivalent to $f$, but
\item $g$ is not right equivalent to $f$.
\end{enumerate}
\end{Def}

\begin{Remark}
The above definition of isomers is an analogue for that for cuspidal 
edges (cf. \cite{HNSUY}).
In the case of cuspidal edges, (1) was replaced by the condition 
that the first fundamental forms of two surfaces are isometric.
(It should be remarked that
all developable surfaces all mutually locally isometric.)
\end{Remark}

We now give a tool to construct
isomers of a given developable strip.
Let $\tilde C$ be an embedded curve in $\R^3$ which 
is homeomorphic to $C$ and has the same total length as $C$.

\begin{Definition}
Let $\tilde{\mb c}(s)$ ($s\in J$) be the arc-length parametrization
of $\tilde C$, and let $\tilde \kappa(s)$ be its curvature function.
Then $\tilde {\mb c}$ is said to be {\it compatible} to $f\in \mc D(C)$
if it satisfies
\begin{equation}\label{eq:CM}
|\hat \mu_f(s)|<\tilde \kappa(s)\qquad (s\in J),
\end{equation}
where  $\hat \mu_f$ is the normalized geodesic curvature function
of $f$.
\end{Definition}

Let $J_0$ be a set which is homeomorphic to $J$.
The following proposition plays an essential role in
considering the relationship of developable strips and 
curved foldings.

\begin{Proposition}\label{prop:Equi}
Let $f:U\to \R^3$ be a developable strip 
belonging to $\mc D(C)$
satisfying $J_0\times \{0\}\subset U$ and $f(J_0\times \{0\})=C$.
Let $u_0\in J_0$ be the point such that $f(u_0,0)=\mb x_0$,
where $\mb x_0$ is the base point $($cf.~\eqref{eq:Bpt}$)$
 of $C$.
Suppose that
$\tilde{\mb c}(s)$ $(s\in J)$ is the arc-length parametrization
of $\tilde C$ 
which is compatible to $f$.
Then  there exist
a tubular neighborhood $V(\subset U)$ of $J_0\times \{0\}$
in $J_0\times \R$ and developable strips
$g_{+},g_-$ belonging to $\mc D(\tilde C)$ 
such that 
\begin{enumerate}
\item $g_+(u_0,0)=g_-(u_0,0)=\tilde {\mb c}(0)$
$($this condition is automatically satisfied if $J_0$
is a closed bounded interval$)$,
\item 
$
J_0\ni u \mapsto g_+(u,0)=g_-(u,0)\in \R^3
$ 
gives a parametrization of $\tilde C:=\tilde {\mb c}(J)$,
\item the first angular functions $\alpha_\pm(u)$ of $g_\pm(u,v)$
satisfy $\alpha_+=-\alpha_-$,
\item $\mu_f=\mu_g$ on $J_0$ $($cf. \eqref{eq:mu-f}$)$, 
\item $\alpha_f(u)$ has the same sign as
$\alpha_{g_+}(u)$ for each $u\in J_0$, and
\item $g_+$ and $g_-$ are normal forms if 
the same is true of
$f$.
\end{enumerate}
\end{Proposition}

\begin{proof}
Since $C$ has total arc-length $l$, 
we can take the arc-length parametrization
$u=u(s)$ ($s\in J$) so that $u_0=u(0)$ and
${\mb c}(s):=\mb c_f(u(s))$ ($s\in J$)
parametrizes $C$.
Since $\tilde {\mb c}$ is compatible to $f$,
$
\tilde {\mb c}(u):=\tilde {\mb c}(s(u))
$ $(u\in J_0)$
gives a parametrization of $\tilde C$ 
defined on $J_0$
such that
\begin{equation}\label{eq;star-d}
\mu_f(u)<\tilde\kappa(u) \qquad (u\in J_0),
\end{equation}
where $\tilde \kappa(u)$ is the curvature function 
of $\tilde{\mb c}(u)$.
Then there exists a unique function $\tilde \alpha:J_0\to (-\pi/2,\pi/2)$ 
such that
$$
\tilde \kappa(u)\cos \tilde \alpha(u)=\mu_f(u), \qquad
\tilde \alpha(u)\alpha_{f}(u)\ge 0
$$
for each $u\in J_0$.
By the compatibility of $\tilde{\mb c}$, 
the function $\sin \tilde \alpha$ 
never vanishes.
Thus, we can define the second angular 
function $\tilde\beta:J_0\to (0,\pi)$
so that
$$
\cot \tilde\beta_\pm(u):=
\frac{\tilde\alpha'(u)\pm |\tilde{\mb c}'(u)|\tilde \tau(u)}
{|\tilde{\mb c}'(u)|\tilde\kappa(u)\sin \tilde \alpha(u)},
$$
where $\tilde\kappa(u)$ and $\tilde\tau(u)$ are the curvature 
and torsion functions of $\tilde{\mb c}(u)$ ($u\in J_0$),
respectively.
We set
\begin{align*}
g_\pm(u,v)& :=\tilde {\mb c}(u)+\tilde \xi_\pm(u), \\
\tilde \xi_\pm (u)&:=\cos \tilde \beta_\pm (u)\tilde{\mb e}(u)+
\sin \tilde \beta_\pm (u)
\Big(
\cos \tilde \alpha(t)\tilde{\mb n}(u)\pm \sin 
\tilde\alpha(u)\tilde{\mb b}(u)
\Big),
\end{align*}
where $\tilde{\mb e}$, $\tilde{\mb n}$ and $\tilde{\mb b}$
are the unit tangent vector field, the
unit principal normal vector field and
the unit bi-normal vector field of $\tilde{\mb c}$,
respectively.
Since $u_0=u(0)$, we obtain (1).
It can be easily checked that 
$g_\pm$ satisfy (2)--(5).
Finally, if $f$ is written in a normal form,
then $u=s$ holds for $s\in J$, and
$\tilde c(u)=\tilde c(s)$ is parametrized
by arc-length. In particular, $g_\pm(u,v)=g_\pm(s,v)$
give normal forms.
So (6) is obtained.
\end{proof}

\begin{Corollary}\label{Prop:1}
Let
$
f:U\to \R^3
$
be a developable strip belonging to
$\mc D(C)$, where 
$U$ is a tubular neighborhood of 
$J_0\times \{0\}$ in $J_0\times \R$.
Then there exist
a tubular neighborhood $V(\subset U)$ of $J_0\times \{0\}$
in $J_0\times \R$ and a developable strip
$
g:V\to \R^3
$
such that
\begin{enumerate}
\item $f(u,0)=g(u,0)$ for each $u\in J_0$,
\item $\alpha_g(u)=-\alpha_f(u)$ for each $u\in J_0$,
\item $g$ is uniquely determined from $f$ as a strip germ along $C$, and
\item $g$ is a normal form if so is $f$.
\end{enumerate}
\end{Corollary}

We call this $g$ the {\it dual} of $f$ and
denote it by $\check f$.
Then an involution
$
\mc D(C)\ni f \mapsto \check f\in \mc D(C)
$
is induced.
Moreover, we have
\begin{equation}\label{eq:D-dual}
\alpha_{\check f}(u)=-\alpha_{f}(u)\qquad (u\in J_0).
\end{equation}

\begin{proof}[Proof of Corollary \ref{Prop:1}]
Setting $\tilde {\mb c}(u):=f(u,0)$, we 
can apply Proposition \ref{prop:Equi} because $f\in \mc D(C)$.
Then the absolute value of the first angular function of $g_-$
coincides with that of $f\in \mc D(C)$, but
the sign is opposite. Thus, $g_-$ gives the desired developable strip.
\end{proof}

By (2) of Corollary \ref{Prop:1},
$\check f$ has the same geodesic curvature function
as $f$, and we obtain the following:

\begin{Proposition} \label{fact:fm1} 
For each $f\in \mc D(C)$, the dual $\check f$ is an isomer of $f$.
\end{Proposition}

Moreover, we have the following:

\begin{Proposition}\label{fact:f}
Let $F,G\in \mc D(C)$ be normal forms satisfying
$F(0,0)=G(0,0)$.
If $\mu_F=\mu_G$
and $F(0,0)=G(0,0)$, then either $G=F$ or $G=\check F$ holds.
\end{Proposition}

\begin{proof}
Since $\mu_F=\mu_G$,
we have $\cos \alpha_F=\cos \alpha_G$, which implies
$\alpha_F=\pm \alpha_G$.
Since the normal form of a developable strip is determined by 
its  first angular function and its base point, 
we have $G=F$ or $G=\check F$.
\end{proof}

\begin{Cor}\label{cor:Tf}
Let $F\in \mc D(C)$ be a normal form.
If $T$ is a symmetry of $C$,
then $T\circ \check F$ is also a normal form giving
the dual of $T\circ F$.
\end{Cor}

\begin{proof}
Since $F$ is a normal form, so is $\check F$.
By Proposition \ref{L0}, $T\circ \check F$ also gives
a normal form. So we have
$$
\mu_{T\circ F}^{}=\mu_{F}^{}
=\mu_{\check F}^{}=\mu_{T\circ \check F}^{}
$$
and $T\circ F(0,0)=T\circ \check F(0,0)$.
Then Proposition \ref{fact:f} implies
$T\circ \check F$ coincides with
$T\circ F$ or its dual.
Since $T\circ \check F(\Omega_\epsilon)$ meets
$T\circ F(\Omega_\epsilon)$ only along $C$
(cf. \eqref{eq:check}), 
we obtain the conclusion.
\end{proof}

\begin{figure}[ht]
\begin{center}
\includegraphics[width=4.8cm]{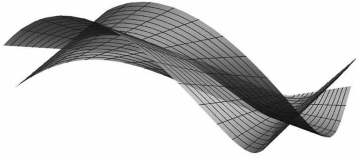}
\caption{The images of $F$ and its dual
given in Example \ref{ex:42}.}\label{fig:NUY}
\end{center}
\end{figure}

\begin{Example}\label{ex:42}
We fix a positive number $l>0$. Then
$$
\mb c_1(s)=\left(
\cos \left(\frac{s}{\sqrt{2}}\right), 
\sin \left(\frac{s}{\sqrt{2}}\right),
\frac{s}{\sqrt{2}}\right) \qquad \left(|s|\le \frac{l}{2}\right)
$$
gives a helix with arc-length parameter satisfying $\kappa=\tau=1/2$.
We set $C_1:=\mb c_1([-l/2,l/2])$.
We fix a constant $\alpha\in (-\pi/2,0)\cup (0,\pi/2)$
and let $F\in \mc D(C_1)$
be a normal form whose first angular function is identically equal to $\alpha$.
Then the dual $\check F$ of $F$ is
obtained by the $180^\circ$-rotation with respect to
the normal line of $C_1$ at the origin.
Figure~\ref{fig:NUY}
shows the images of $F$ and $\check F$.
\end{Example}

We consider the case that
$J=\mathbb I_l$.
Regarding
Lemma \ref{Aeq1}, we give the following:

\begin{Def}
The image $\Gamma$ of a regular curve $\gamma:\mathbb I_l\to \R^2$ 
parametrized by arc-length is called a {\it generator} of a
strip $f\in \mc D(C)$ ($C:=\mb c(\mathbb I_l)$)
if the curvature function of $\gamma(s)$
is equi-affine equivalent to the normalized geodesic curvature function 
$\hat \mu_{f}(s)$ ($s\in \mathbb I_l$) of $f$.
\end{Def}

We next consider the case that
$J=\mathbb T^1_l$.
Regarding
Lemma \ref{Aeq2}, we give the following:

\begin{Def}
The image $\Gamma$ of a regular curve 
$\gamma:\mathbb I_l\to \R^2$ parametrized by arc-length is 
called a {\it generator} of the 
strip $f\in \mc D(C)$ ($C:=\mb c(\mathbb T^1_l)$)
if the curvature function of $\gamma(s)$
can be smoothly extended as 
an $l$-periodic function 
which is the lift of a function which
is equi-affine equivalent to
the normalized curvature function
$\hat \mu_{f}(s)$ ($s\in  \mathbb T^1_l$) of $f$.
\end{Def}

A generator $\Gamma$ of $f$ has an ambiguity of isometric motions
in the plane $\R^2$. Since $f$ is a developable surface,
it can be developed to a plane, and the curve $C$ is deformed
to a plane curve which is congruent to $\Gamma$.
If $\Gamma$ has no self-intersections,
then $f$ can be obtained as a deformation of
a developable strip from a tubular neighborhood
of $\Gamma$ to the image of $f$.
According to \cite{FT0}, we define \lq\lq origami maps''
as follows:

\begin{Def}\label{def:p-ps}\rm
The {\it origami map} 
$\Phi_f$ induced by  $f\in \mc D(C)$
is defined by
$$
\Phi_f(u,v):=
\begin{cases}
f(u,v) & (v\ge 0), \\
\check f(u,v) & (v<0), 
\end{cases}
$$
where $\check f$ is the dual of $f$.
In this setting, $C$ is called the {\it crease} of $\Phi_f$,
and a generator $\Gamma$ of $f$ 
is called a {\it crease pattern} of the origami map $\Phi_f$.  
(Since the normalized curvature function $\mu$
is common in $f$ and $\check f$, $\Gamma$ gives a 
generator of $\check f$. When $J_0$ is a bounded closed 
interval, the congruence class of
the crease pattern of $\Phi_f$
is uniquely determined.)
\end{Def}

We set
$$
\mc O(C):=\{\Phi_f\,;\, f\in {\mc D}(C)\},\quad
\mc O_*(C):=\{\Phi_f\,;\, f\in {\mc D}_*(C)\}
$$
and
$$
\mc O(|C|):=\mc O(C)\cup \mc O(-C),\quad
\mc O_*(|C|):=\mc O_*(C)\cup \mc O_*(-C),
$$
which are the sets of origami maps and
the sets of {\it admissible origami maps}
along $C$ and $|C|$, respectively.
Figure \ref{fig:A} indicates 
the second angular functions $\beta$ and $\check \beta$ 
of $f$ and $\check f$, respectively. 
Here, 
$$
\check \Phi_{f}:=\Phi_{\check f}
$$
is called the {\it adjacent origami map}
with respect to $\Phi_f$. 
Obviously, the union of the images of $\Phi_f$ and $\check \Phi_f$
coincides with the union of the images of $f$ and $\check f$.

\begin{figure}[htb]
\begin{center}
 \begin{center}
   \includegraphics[width=0.3\linewidth]{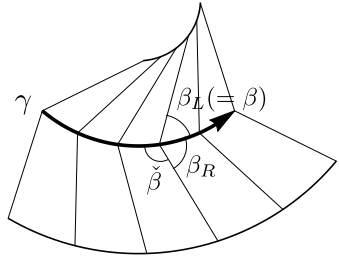} 
 \end{center}
\caption{Angular functions $\beta_L$ 
and $\beta_R$
on the crease pattern of $\Phi_f$.}\label{fig:S}
\end{center}
\end{figure}

The following fact is known:

\begin{Fact}[Fuchs-Tabachnikov \cite{FT0}, see also \cite{Ori}]
\label{1}
A curved folding $P$ along a curve $|C|$ 
satisfying {\rm (i),(ii),(iii)} and {\rm (iv)}
$($resp. {\rm (iv${}'$)}$)$ in the introduction
is realized as 
the image of $\Phi_F$ for a certain normal form $f\in {\mc D}(C)$
$($resp. $f\in {\mc D}_*(C))$.
Moreover, $\Gamma$ corresponds to the crease pattern
of $\Phi_f$.
\end{Fact}

\begin{Remark}
More precisely,
the condition (ii) in the introduction
implies $\cos \alpha_F >0$ and
(iv) (resp. (iv$'$)) in the introduction corresponds to
the condition $\cos \alpha_F \ne 1$
(resp. $\dy\max_{u\in J_0}|\cos \alpha_F(u)| <1$).
Moreover, $\mu_F$ coincides with the curvature function 
with respect to the
arc-length parametrization 
$\gamma:J\to \R^2$ of the generator $\Gamma$.
We set $\beta_L:=\beta$ and 
$\beta_R:=\pi-\check \beta$,
where $\beta$ and $\check \beta$ are the second angular 
functions of $F$ and $\check F$, respectively.
Then $\beta_L$ (resp. $\beta_R$) gives the left-ward (resp. right-ward)
angular function of the ruling direction from the tangential direction
$\gamma'(s)$ in the plane $\R^2$,
see Figure \ref{fig:S}. 
\end{Remark}

Here, we set
\begin{align*}
\mc O(\Gamma,|C|)&:=\{\Phi_f\,; \, \text{the generator of 
$f\in {\mc D}(|C|)$ is $\Gamma$}\}, \\
\mc O_*(\Gamma,|C|)&:=\{\Phi_f\,;\, 
\text{the generator of $f\in {\mc D}_*(|C|)$ is $\Gamma$}\}
\end{align*}
and
\begin{equation} \label{eq:o2}
\mc P(\Gamma,C):=\{\op{Im}(\Phi_f)\,; \, f\in \mc O(\Gamma,C)\},\quad
\mc P_*(\Gamma,C):=\{\op{Im}(\Phi_f)\,; \, f\in \mc O_*(\Gamma,C)\},
\end{equation} 
where $\op{Im}(\Phi_f)$ is the image of the strip germ $\Phi_f$ along $C$.
Then $\mc P(\Gamma,C)$ can be considered as
the set of curved foldings whose crease and crease pattern are $C$
and $\Gamma$, respectively. 
The set $\mc P_*(\Gamma,C)$ consists of admissible curved foldings
along $C$ defined as in the introduction.
The following assertion is obvious by the definition of
the map $\Phi:{\mc D}(|C|)\to \mc O(|C|)$ and 
Lemma~\ref{lem:L}.

The following assertion implies that
we can fold a pair 
$
\Phi_F(\Omega_\epsilon),\,
\check \Phi_F(\Omega_\epsilon)
$
of curved foldings at the same time whenever 
$\epsilon(>0)$ is sufficiently small:

\begin{Proposition}\label{prop:220}
For each normal form $F\in {\mc D}(|C|)$, 
\begin{equation}\label{eq:check}
F(\Omega_\epsilon)\cap \check F(\Omega_\epsilon)=C
\end{equation}
holds. Moreover,  the induced origami map $\Phi_F$ has
no self-intersections.
\end{Proposition}

\begin{proof}
Since the ruling vector of $F$
is linearly independent of $\check F$
at each point of $C$,
the first assertion follows from 
Proposition \ref{lem:key}.
So we prove the second assertion.
If not, \eqref{eq:check} implies
that $F(\Omega_\epsilon^+)$ must 
meet $F(\Omega_\epsilon^-)$,
which is impossible.
\end{proof}

\begin{Proposition}\label{Prop:Phi1}
Let $F,G\in {\mc D}(|C|)$ be normal forms.
If $\Phi_F(\Omega_{\epsilon})$
coincides with $\Phi_G(\Omega_{\epsilon})$
for sufficiently small $\epsilon$,
then $F$ is right equivalent to $G$.
\end{Proposition}

\begin{proof}
We denote by $F,G$ the normal forms associated with $f,g$ 
satisfying $F(0,0)=G(0,0)$,
respectively. 
It is sufficient to show the assertions hold for $F$ and $G$.
We suppose $\Phi_F(\Omega_{\epsilon})=\Phi_G(\Omega_{\epsilon})$.
By replacing $F$ (resp. $G$) by $F^\sharp$ (resp. $G^\sharp$), 
Proposition \ref{L0} yields that
$F,G\in {\mc D}(C)$ without loss of generality.
Then we have
$
F(\Omega^+_\epsilon)\cup \check F(\Omega^-_\epsilon)
=G(\Omega^+_\epsilon)\cup \check G(\Omega^-_\epsilon).
$
By Proposition \ref{lem:pm},
$F(\Omega^+_\epsilon)\cap \check G(\Omega^-_\epsilon)$
is the empty set, and so we have
$F(\Omega^+_{\epsilon})=G(\Omega^+_{\epsilon})$.
By Proposition \ref{fact:1}, we can conclude that
$F=G$. 
\end{proof}

We now the following:

\begin{Theorem}\label{prop:introfunctor}
The map 
$
\Phi:{\mc D}(|C|) \ni f \mapsto \Phi_f\in \mc O(|C|)
$
has the following properties:
\begin{enumerate}
\item $f,\,g\in {\mc D}(|C|)$ are right equivalent if and only if
$\Phi_F(\Omega_\epsilon)$ coincides with
$\Phi_G(\Omega_\epsilon)$,
where $F,G$ are normal forms associated with $f$ and $g$ 
satisfying $F(0,0)=G(0,0)=\mb x_0$, respectively.
\item If $\Phi_f$ and $\Phi_g$ $(f,\,g\in {\mc D}(|C|))$
have the same crease pattern,
then $f$ and $g$ are geodesically equivalent.
\item For each $f\in {\mc D}(|C|)$,
      the crease pattern of $\check \Phi(f)(=\Phi(\check f))$
      coincides with that of $\Phi(f)$.
\item Let $T$ be a symmetry of $C$,
then $T\circ \Phi_F=\Phi_{T\circ F}$ holds for each 
normal form $F\in {\mc D}(|C|)$.
\item Let $f,\,g\in {\mc D}(|C|)$.
If $\Phi_F(\Omega_\epsilon)$ 
is congruent to $\Phi_G(\Omega_\epsilon)$,
then there exists a symmetry $T$ of $C$ such that
$g$ is right equivalent to $T\circ f$.
\end{enumerate}
\end{Theorem}

Before proving this assertion,
we prepare the following lemma:

We prepare the following:

\begin{Lemma}\label{lem:L}
Let $C$ be a non-closed space curve 
$($i.e. $J_0$ is
a bounded closed interval$)$, and
let $\Gamma$ be a simple closed arc in $\R^2$
which is a generator of a developable strip $F\in {\mc D}(C)$
written in a normal form.
Then the following two assertions are equivalent:
\begin{itemize}
\item $\Gamma$ has a symmetry $($cf. Definition \ref{def:S}$)$,
\item the geodesic curvature function $\mu_F$ 
of $F$ has a symmetry.
\end{itemize}
\end{Lemma}

\begin{proof}
Since $\hat \mu_F$ coincides with the curvature function of $\Gamma$,
the conclusion follows from  Lemma~\ref{Aeq1}.
\end{proof}

\begin{proof}[Proof of Theorem \ref{prop:introfunctor}]
We denote by $F,G$ the normal forms associated with $f,g$
satisfying $F(0,0)=G(0,0)$,  respectively. 
It is sufficient to show the assertions hold for $F$ and $G$.
Then $F$ is defined on a tubular
neighborhood of $J\times \{0\}$ in $J\times \R$.
If $F,G$ are right equivalent, then 
it is obvious that the images of $\Phi_F,\Phi_G$ coincide.
The converse of this assertion follows 
from Proposition~\ref{Prop:Phi1}.

We now prove (2).
In the case that $J$ is a bounded closed interval,
 (2) follows from Lemma~\ref{lem:L}.
We then consider the case that 
$J$ is a one dimensional torus $\mathbb T_l$ ($l>0$).
If $\Phi_F$ and $\Phi_G$ have the 
same crease pattern $\Gamma$, then the
curvature function of $\Gamma$ can be 
extended as a smooth $l$-periodic function on $\R$,
and coincides with the lift of common normalized
geodesic curvature function of $F$ and $G$.
So $F$ and $G$ are geodesically equivalent.

On the other hand, (3) is obvious from the definition of
$\check \Phi(F)$.

We next prove (4). Let $T$ be a symmetry of $C$,
then, by Corollary \ref{cor:Tf}, we have
$$
T\circ \Phi_F(\Omega_\epsilon)=
T\circ F(\Omega^+_\epsilon)\cup 
T\circ \check F(\Omega^-_\epsilon)=\Phi_{T\circ F}(\Omega_\epsilon).
$$
So $T\circ \Phi_F$ is right equivalent to $\Phi_{T\circ F}$
by (1). Since $F$ is a normal form, we have
$T\circ \Phi_F=\Phi_{T\circ F}$.

Finally, we prove (5).
Suppose that $\Phi_F(\Omega_\epsilon)$ 
is congruent to $\Phi_G(\Omega_\epsilon)$.
Then there exists an isometry $T$ of $\R^3$ such that
$T\circ \Phi_G(\Omega_\epsilon)$ coincides with
$\Phi_F(\Omega_\epsilon)$.
Since $T\circ \Phi_G(\Omega_\epsilon)=\Phi_{T\circ G}(\Omega_\epsilon)$,
Proposition \ref{Prop:Phi1} implies that
$T\circ G$ is right equivalent to $F$.
So we obtain (5).
\end{proof}

\section{The inverses and inverse duals}

In this section, we set $J=\mathbb I_l(=[-l/2,l/2])$ ($l>0$), 
that is, $C$ is a non-closed space curve.
Let $I:=[b,c]$ ($b<c$)
be a closed bounded interval on $\R$.

\begin{Proposition}\label{cor:C2}
Let $f:U\to \R^3$ be a developable strip belonging to $\mc D_*(C)$,
where $U$ is a tubular neighborhood of 
$I\times \{0\}\, (\subset I\times \R)$.
Then there exist a tubular 
neighborhood $V(\subset U)$ of $I\times \{0\}$
in $I\times \R$ and two maps
$
f_*,\check f_*:V\to \R^3
$
such that
\begin{enumerate}
\item $f_*$ and $\check f_*$ belong to $\mc D_*(-C)$,
\item $f_*(u,0)=\check f_*(u,0)=f(-u,0)$ for each $u\in I$,
\item the first angular function $\alpha_{f_*}$
takes the same sign as $\alpha_f$ and
satisfies
\begin{equation}\label{eq:as}
\kappa_f(b+c-u)\cos \alpha_{f_*}(u)
=\kappa_f(u)\cos\alpha_{f}(u), 
\end{equation}
where
$\mb c_f(u):=f(u,0)$ $(u\in I)$ and
 $\kappa_f(u)$ is its curvature function,
\item $\alpha_{\check f_*}(u)=-\alpha_{f_*}(u)$,
\item $f_*$ and $\check f_*$ are normal forms
if $f$ is.
\end{enumerate}
Moreover, such two maps $f_*$ and $\check f_*$
are uniquely determined from $f$ as map germs. 
\end{Proposition}

We call $f_*$ the {\it inverse} of $f$, and
$\check f_*$ the {\it inverse dual} of $f$ (cf. \cite{Ori}).
By definition,
\begin{equation}\label{eq:star}
\mu_{f_*}(u)=\mu_{\check f_*}(u)=\mu_{f}(u),
\end{equation}
and so each generator of $f$ gives
a generator of $f_*$ (and of $\check f_*$). 

\begin{proof}
Since $f$ is admissible (i.e. $f\in \mc D_*(C)$), we have
$$
0<\kappa_f(u)\cos \alpha_f(u)<\min_{u\in I}\kappa_f(u) 
\qquad (u\in I).
$$
We note that ${\mb c}^\sharp(u):=\mb c_f(b+c-u)$ ($u\in I$)
gives the parametrization of $\tilde C:=-C$.
Then $\kappa_f(b+c-u)$ is the curvature function of ${\mb c}^\sharp(u)$.
Since $\mu_f:=\kappa_f\cos \alpha_f$
is the curvature function of the generator of $f$,
we have
\begin{equation}\label{eq:Cm}
0<\mu_f(u)<\min_{u\in I}\kappa_f(u)=
\min_{u\in I}\kappa_f(b+c-u) \qquad (u\in I).
\end{equation}
So ${\mb c}^\sharp$ is compatible to $f$.
Thus,
there exists a developable strip
$g_{+}\in {\mc D_*(-C)}$ $($resp. $g_{-}
\in {\mc D_*(-C)})$ satisfying (1)--(6) of Proposition \ref{prop:Equi}.
In particular, the first angular function of $g_+$ (resp. $g_-$)
is positive (resp. negative).
Moreover, by \eqref{eq:Cm}, $g_+$ and $g_-$ are belonging to $\mc D_*(-C)$.
Then it can be easily checked that $f_*:=g_+$ and $\check f_*:=g_-$ satisfy
(1)--(5) of Proposition \ref{cor:C2}.
In fact, 
$$
\mu_{f_*}(u)=\kappa_f(b+c-u)\cos \alpha_{f_*}(u)
$$
coincides with $\mu_f(u)(=\kappa_f(u)\cos \alpha_f(u))$
by (2) of Proposition \ref{prop:Equi}.
\end{proof}

\begin{Remark} Fix $f\in \mc D_*(C)$.
By \eqref{eq:alpha-s}, the first angular function of 
the inverse dual $f_*$ takes the opposite sign as that of $f^\sharp$.
\end{Remark}

We have the following:

\begin{Proposition}
If $g\in \mc D_*(|C|)$ 
is geodesically equivalent to $f\in \mc D_*(C)$, then
$g$ is right equivalent to one of
$\{f,\,\,\check f,\,\, f_*,\,\,\check f_*\}$.
\end{Proposition}

\begin{proof}
Replacing $g$ by $g^\sharp$, we may assume that
$f,g\in D_*(|C|)$.
We denote by $F,G$ the normal forms associated with $f,g$
satisfying $F(0,0)=G(0,0)$,  respectively. 
Then the assumption that
$g$ is geodesically equivalent to $f$
implies $\mu_F=\mu_G$.
By Proposition \ref{fact:f} we have
$G=F$ or $G=\check F$.
On the other hand, if
$G\in \mc D_*(-C)$, then, replacing $C$ by $-C$ and
applying Proposition \ref{fact:f},
we can conclude $G=F_*$ or $G=\check F_*$.
\end{proof}

Moreover, we can prove the following:

\begin{Proposition}\label{prop:star00}
For each normal form $F\in \mc D_*(C)$, 
the ruling direction of 
the inverse $F_*$ is linearly independent of
that of $F$ at each point of $C$. 
In particular, 
\begin{equation}\label{eq:star}
F(\Omega_\epsilon)\cap F_*(\Omega_\epsilon)=C
\end{equation}
holds for each sufficiently small $\epsilon(>0)$.
\end{Proposition}

\begin{proof}
Without loss of generality, we may replace $f$
by its normal form $F$.
Since the first angular function of $(F_*)^\sharp$
has the opposite sign of that of $F$ (cf. \eqref{eq:alpha-s}),
Lemma~\ref{lem:Tv} yields that the ruling direction of
$F_*$ is linearly independent of that of $F$ along $C$. 
The last assertion is a consequence of Proposition \ref{lem:key}.
\end{proof}

Using this proposition, we can prove the following:

\begin{Proposition}\label{prop:ffs}
For $f\in \mc D_*(C)$, the following three assertions
are equivalent:
\begin{enumerate}
\item $f$ is right equivalent to $\check f_*$,
\item $\check f$ is right equivalent to $f_*$,
\item the normalized geodesic
curvature $\hat \mu_f$ of $f$
has a symmetry $($cf. Definition~\ref{ft-sym}$)$.
\end{enumerate}
\end{Proposition}

\begin{proof}
The equivalency of (1) and (2) is obvious.
So it is sufficient to show that (1)
is equivalent to (3).
We may assume that $f$ is a normal form and
denote it by $F$. 
Then $F(s,v)$ is defined on a 
tubular neighborhood of $\mathbb I_l\times \{0\}$
in $\R^2$, where $l$ is the total arc-length of $C$.
Suppose that $\mu_F$ has a symmetry.
Since $\mu_F(s)=\mu_F(-s)$ for $s\in \mathbb I_l$,
$F^\sharp$ is geodesically equivalent to $F$.
Since $F^\sharp\in \mc D_*(-C)$,
Proposition \ref{fact:f} yields that 
$F^\sharp$ coincides with $F_*$ or $\check F_*$.
However, $F^\sharp$ never coincides with $F_*$
by Proposition \ref{prop:star00}.
So we have $F^\sharp=\check F_*$.

Conversely, we suppose (1). Then
$F^\sharp=\check F_*$ holds.
By \eqref{eq:star},
we have
$$
\mu_F(-s) = 
\mu_{F^\sharp}(s)=\mu_{\check F_*}(s)=\mu_{F}(s),
$$
which implies that $\mu_F$ has a symmetry.
\end{proof}

In the above discussions, the following assertion was also obtained.

\begin{Corollary}\label{cor:F_sharp}
Let $F\in {\mc D_*}(C)$ be  a normal form.
If the geodesic curvature $\mu_F$
has a symmetry, then $\check F_*=F^\sharp$ holds.
\end{Corollary}

\begin{Theorem}\label{thm:2G}
Let $f\in \mc D_*(C)$ and $n_f$ the number of right 
equivalence classes
of $f,\,\check f,\, f_*$ and $\check f_*$.
If the normalized geodesic
curvature of $f$ has no symmetries, then $n_f=4$,
otherwise $n_f=2$.
\end{Theorem}

\begin{proof}
We may assume that $f$ is a normal form and
denote it by $F$. Suppose that $\mu_F$ has a symmetry.
By Proposition \ref{prop:ffs}, we have
$F=\check F_*$ and $\check F=F_*$,
and $n_F=2$.
On the other hand, suppose that $n_f<4$.
If necessary, replacing $F$ by one of $\{\check F$, $F_*, \check F_*\}$,
we may assume that $F$ coincides with
one of  $\check F, F_*, \check F_*$.
By
\eqref{eq:check} and \eqref{eq:star},
$F$ must coincide with
$\check F_*$.
By Proposition~\ref{prop:ffs}, $\mu_F$ has a symmetry.
\end{proof}

\begin{Remark}
In Example \eqref{ex:42},
the curvature function of $C$ and
the angular function of $F$ are constant,
$\check F_*=F^\sharp$ and $\check F=F_*^\sharp$
hold. So in this case,$n_f=2$ holds.
\end{Remark}

\medskip
\begin{proof}[Proof of Theorem A] 
We fix a curved folding
$P\in \mc P_*(\Gamma,C)$ 
arbitrarily.
Then there exists a normal form $F\in \mc D_*(C)$
such that $P=\Phi_F$.
Moreover,
$$
\Phi_F(\Omega_{\epsilon}),\,\, \Phi_{\check F}(\Omega_{\epsilon}),\,\,
\Phi_{F_*}(\Omega_{\epsilon}),\,\, \Phi_{\check F_*}(\Omega_{\epsilon})
$$
produce all candidates of curved foldings.
We let $\Gamma$ be a generator of $F$.
Since $\Gamma$ has no self-intersections,
the symmetries of $\Gamma$ correspond to the symmetries
of the geodesic curvature function $\mu_F$
(cf. Lemma \ref{lem:L}).
By Proposition \ref{Prop:Phi1},
the number of elements in $\mc P_*(\Gamma,C)$ coincides with
the number of distinct subsets in
$
\{F(\Omega_{\epsilon}),\check F(\Omega_{\epsilon}), 
F_*(\Omega_{\epsilon}), \check F_*(\Omega_{\epsilon})\}.
$
So we obtain Theorem A by Theorem \ref{thm:2G}.
\end{proof}

\begin{figure}[htb]
\begin{center}
 \begin{center}
   \includegraphics[width=0.25\linewidth]{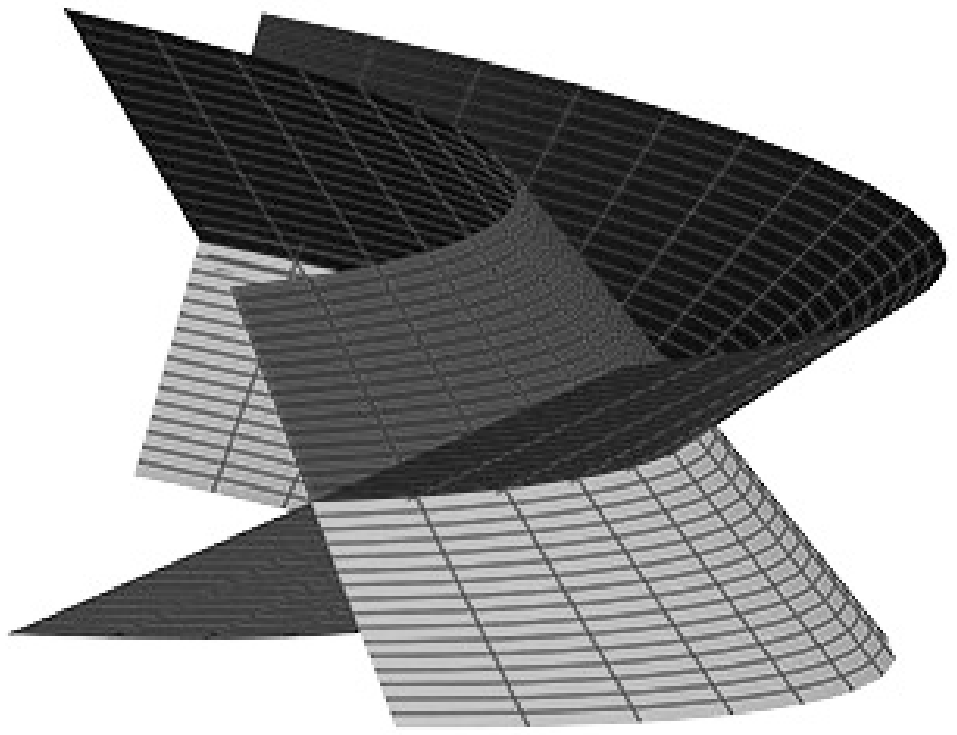} \qquad \quad
   \includegraphics[width=0.25\linewidth]{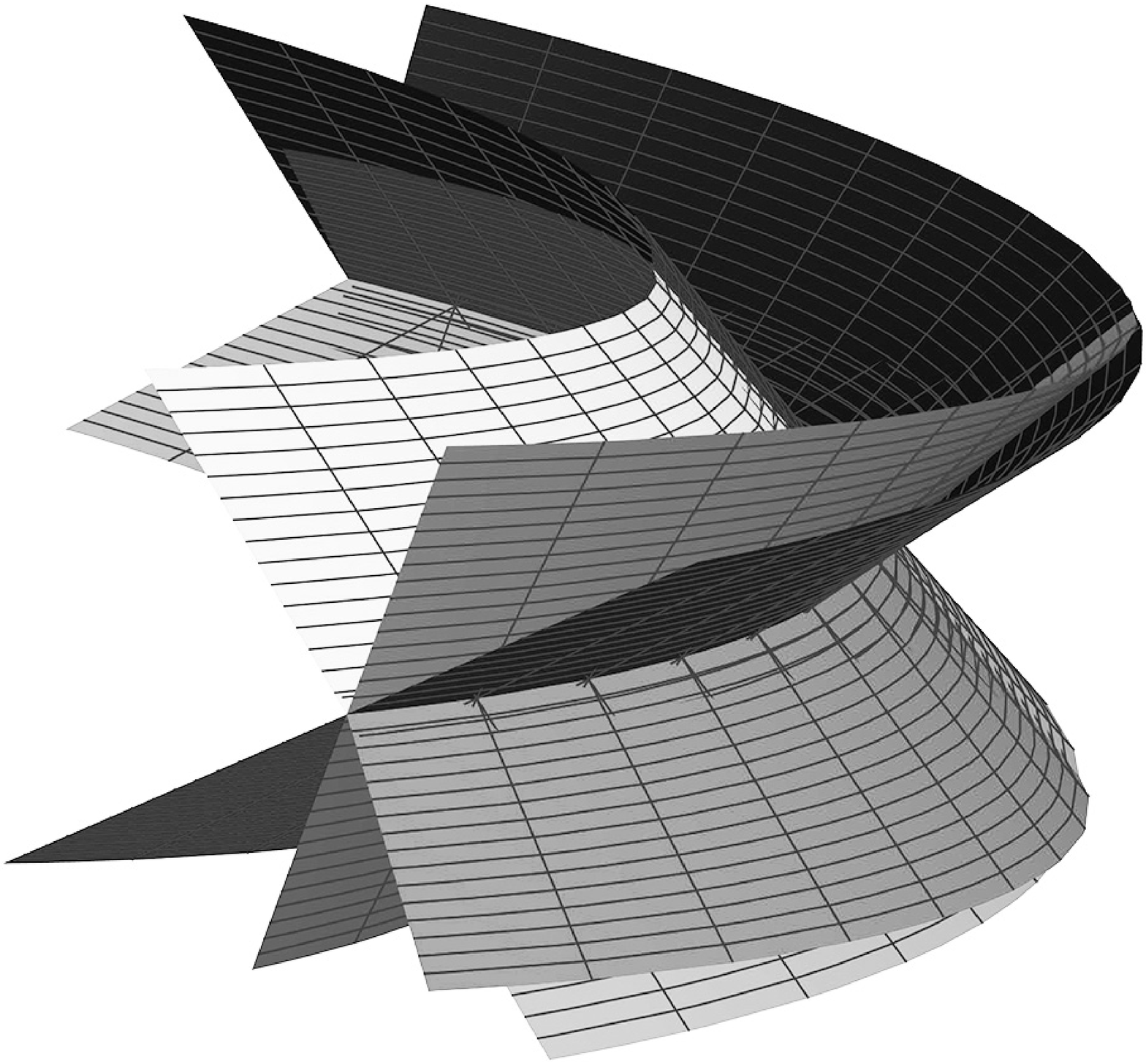} 
  \end{center}
\caption{Images of $\Phi_1$ and $\Phi_2$ (left)
and images of $\{\Phi_j\}_{j=1,\dots,4}$ (right).
}\label{fig:two}
\end{center}
\end{figure}

\begin{Example}
We consider a one-quarter of the unit circle given by
$$
\mb c(s):=(\cos s, \sin s, 0) \qquad \left(|s|\le \frac{\pi}{4}\right)
$$
and
$$
\alpha(s):=\frac{\pi}4-\frac{s}{2},\qquad 
T:=\pmt{1 & 0& 0 \\ 0 & -1& 0 \\ 0 & 0 & 1 }.
$$
Then $T(C)=C$, and 
the developable strip $F:=F^{\alpha}$
along $C:=\mb c([-\pi/4,\pi/4])$ 
with the first angular function $\alpha$
induces four associated origami maps
$$
\Phi_1:=\Phi_F,\quad \Phi_2:=\check \Phi_F,
\quad \Phi_3:=T\circ \Phi_F,\quad \Phi_4:=T\circ \check \Phi_F.
$$
Since the normalized geodesic curvature
$
\mu_F=\cos \alpha
$
does not have symmetries, the four curved foldings
are distinct. In fact, 
Figure \ref{fig:two} 
left (resp. right) indicates
the images of $\Phi_1$ and $\Phi_2$ (resp. $\Phi_1,\dots,\Phi_4$)), by which 
we can observe the four curved foldings $\Phi_1,\dots,\Phi_4$
are distinct.
\end{Example}

\section{The congruence classes of isomers of developable strips}

We first consider the case that $C$ admits a symmetry:

\begin{Lemma}\label{prop:Pi}
 Suppose that $C$ lies in a plane $\Pi$, and let $T_0$ be
the reflection with respect to $\Pi$. Then  
$T_0\circ f=\check f$ holds for each $f\in \mc D(C)$. 
\end{Lemma}

\begin{proof}
We may assume that $C$ lies in the $xy$-plane in $\R^3$.
The reflection $T_0$ maps $(x,y,z)\in \R^3$ to 
$(x,y,-z)\in \R^3$. 
Since $\mb b=(0,0,1)$ and
the second angular function of $\check f$ 
coincides with that of $f$, 
the assertion follows by a direct calculation.
\end{proof}

Let $\mb c(s)$ ($s\in \mathbb I_l$) be the arc-length parametrization of $C$,
where $\mathbb I_l:=[-l/2,l/2]$.
The following assertion plays an important role 
in the latter discussions:

\begin{Theorem}\label{prop:24}
Let $F\in {\mc D_*}(C)$ be  a normal form,
and let $T$ be a non-trivial symmetry of $C$.
Then, the following two assertions hold:
\begin{enumerate}
\item If $T$ is a positive symmetry, then 
$T\circ F=F_*$.
Moreover, if $\mu_F$ has a symmetry, then $T\circ F(-s,v)=\check F(s,v)$.
\item If $T$ is a negative symmetry, then  
$
T\circ F=\check F_*.
$
Moreover, if $\mu_F$ has a symmetry, then $T\circ F(-s,v)=F(s,v)$.
\end{enumerate}
\end{Theorem}

We remark that the developable strip $F^\alpha$ given in 
Example \ref{ex:42} satisfies (1).  

\begin{proof}
By a suitable motion in $\R^3$,
we may assume that $F(0,0)=\mb 0$ and
$T$ is an orthogonal matrix.
Then its determinant $\sigma:=\det(T)$
is equal to $1$ (resp. $-1$) if $T$ is a positive (resp. negative)
symmetry of $C$.

Since $C$ admits a non-trivial symmetry,
we have $\kappa(-s)=\kappa(s)$, where $\kappa(s)$
is a curvature function of $C$ with respect to the
arc-length parametrization $\mb c$ of $C$ on $\mathbb I_l$.
By the definition of $\alpha_*(s)$, we have
$$
\kappa(-s)\cos \alpha_*(s)=
\mu(s)=\kappa(s)\cos \alpha_F(s)=\kappa(-s)\cos 
\alpha_{F}(s),
$$
that is,
$
\cos \alpha_*=\cos \alpha_F.
$
By (4) of Proposition \ref{cor:C2},
we can conclude $\alpha_*=\alpha_F$.
Then we have
$$
T\xi_F
=
\cos \beta_F \, T\mb e+\sin \beta_F\Big(
\cos \alpha_* \, T\mb n+\sin \alpha_* \,T\mb b\Big),
$$
where $\beta_F$ is the second angular function of $F$
and $\mb e,\,\mb n,\,\mb b$ are the unit tangent vector
field, the unit principal normal vector field and the
unit bi-normal vector field of $\mb c$, respectively.

 Since (cf. \eqref{eq:s2} and \eqref{eq:TNB}), we have
$$
T \mb n(s)=\mb n(-s)=\mb n^\sharp(s), \qquad
T \mb b(s)=-\sigma \mb b(-s)=\sigma \mb b^\sharp(s)
$$
and
\begin{align*}
T\xi_F(s)
&=
\cos(\pi-\beta_F(s)) \mb e^\sharp(s) \\
&\phantom{*****}+\sin(\pi-\beta_F(s))
\Big(
\cos (\sigma\alpha_*(s))\, \mb n^\sharp(s)+
\sin(\sigma \alpha_*(s)) \, \mb b^\sharp(s)
\Big).
\end{align*}
By
Proposition \ref{L0},
$T\xi_F(s)$
gives the normalized ruling vector field of $T\circ F$.
Thus $T\circ F(s,v)$ belongs to ${\mc D}_*(-C)$,
and its first and second angular functions are given by
$-\sigma\alpha_*(s)$ and $\pi-\beta_F(s)$, respectively.
Since the geodesic curvature function of $T\circ F(s,v)$
coincides with $\hat \mu(s)$ 
(cf. Proposition \ref{L0}),
$T\circ F$ must coincide with $F_*$ (resp. $\check F_*$)
if $T$ is a positive (resp. negative) symmetry of $C$.

We next suppose that $\mu_F$ has a symmetry.
By Proposition \ref{cor:C2}, $F^\sharp=\check F_*$ holds. 
If $\sigma=1$, then
$$
T\circ F(s,v)=F_*(s,v)=\check F^\sharp(s,v)=\check F(-s,v),
$$
which implies the second assertion of (1).
Similarly, considering the case of $\sigma=-1$,
we also obtain the second assertion of (2).
\end{proof}

\begin{Corollary}\label{cor:24}
Let $F\in {\mc D_*}(C)$
be  a normal form.
If $C$ has a non-trivial symmetry $T$,
then  
$\{F(\Omega_{\epsilon}),\check F(\Omega_{\epsilon})\}$
coincide with
$\{T\circ F_*(\Omega_{\epsilon}),T\circ \check F_*(\Omega_{\epsilon})\}$
for sufficiently small $\epsilon(>0)$.
\end{Corollary}

\begin{proof}
By Theorem \ref{prop:24},
if $T$ is positive (resp. negative), then
$
T\circ F(s,v)=F_*(s,v)
$
(resp.
$
T\circ F(s,v)=\check F_*(s,v)
$)
and
$
T\circ \check F(s,v)=\check F_*(s,v)
$
(resp.
$
T\circ \check F(s,v)=F_*(s,v)).
$
\end{proof}

The following assertion is a refinement of \cite[Lemma 3.2]{Ori}:

\begin{Proposition}\label{cor:001}
Let $F\in {\mc D}_*(C)$ be a normal form.
Then $F(\Omega_{\epsilon})$ is congruent to
$\check F(\Omega_{\epsilon})$ 
for sufficiently small $\epsilon>0$ if and only if  
\begin{enumerate}
\item $C$ lies in a plane, or 
\item $C$ has a positive symmetry
and $\mu_F$ also has a symmetry. 
\end{enumerate}
\end{Proposition}

\begin{proof}
By Lemma \ref{prop:Pi} and (2) of Theorem~\ref{prop:24},
(1) or (2) implies that $F(\Omega_{\epsilon})$ is
congruent to $\check F(\Omega_{\epsilon})$.
To show the converse assertion, 
we suppose that $F(\Omega_{\epsilon})$ is
congruent to $\check F(\Omega_{\epsilon})$.
Then, there exists an isometry $T$ on $\R^3$
such that 
\begin{equation}\label{eq:AAA}
T\circ \check F(\Omega_{\epsilon})
=F(\Omega_{\epsilon}).
\end{equation}
By Lemma \ref{prop:Pi},
we may assume that $C$ does not lie in any plane.
It is sufficient to show (2).
Since $T$ must be non-trivial, that is, it reverses 
the orientation of $C$ (cf. Proposition~\ref{prop:03a}).
If $T$ is a negative symmetry, then
by Theorem~\ref{prop:24}, we have
$$
\check F(\Omega_{\epsilon})=
T\circ F(\Omega_{\epsilon})
=T\circ F^\sharp(\Omega_{\epsilon})=F_*(\Omega_{\epsilon}).
$$
By Proposition \ref{fact:1},
we have $\check F=F_*$, and
by Proposition \ref{prop:ffs}, $\Gamma$ has a symmetry.
Then, by Theorem~\ref{prop:24}, we have
$$
\check F(\Omega_{\epsilon})=
T\circ F(\Omega_{\epsilon})=F(\Omega_{\epsilon}),
$$
contradicting \eqref{eq:check}.
So $T$ is
a positive symmetry and
$T\circ F=F_*$ by Theorem~\ref{prop:24}.
By \eqref{eq:AAA}, we have
$$
F_*(\Omega_{\epsilon})=
T\circ F(\Omega_{\epsilon})
=\check  F(\Omega_{\epsilon}),
$$
which implies $F_*=\check F$.
By  Proposition \ref{prop:ffs},
$\mu_F$ has a symmetry.
So we obtain (2).
\end{proof}

\begin{Proposition}\label{cor:00a} Let 
$F\in {\mc D}_*(C)$ be a normal form.
Then, for sufficiently small $\epsilon (>0)$,
$F(\Omega_{\epsilon})$ is
congruent to $\check F_*(\Omega_{\epsilon})$
if and only if 
\begin{enumerate}
\item $C$ has a negative symmetry, or 
\item $\mu_F$ has a symmetry. 
\end{enumerate}
\end{Proposition}

\begin{proof}
By (2) of Theorem \ref{prop:24} and
Corollary \ref{cor:F_sharp},
(1) or (2) implies that
$F(\Omega_{\epsilon})$ is
congruent to $\check F_*(\Omega_{\epsilon})$.
So it is sufficient to show the converse.
We suppose that $F(\Omega_{\epsilon})$ is
congruent to $\check F_*(\Omega_{\epsilon})$.
We also suppose that
$\mu_F$ has no symmetries.
By Proposition~\ref{prop:ffs},
$F(\Omega_{\epsilon})\ne 
\check F_*(\Omega_{\epsilon})$ holds.
So $C$ must have a symmetry $T$ such that
$
T\circ F(\Omega_{\epsilon})
=\check F_*(\Omega_{\epsilon}).
$
If $T$ is not non-trivial, then $C$ lies in a plane $\Pi$
and $T$ is the reflection with respect to $\Pi$.
By Lemma \ref{prop:Pi}, we have 
$$
\check F(\Omega_{\epsilon})=T\circ F(\Omega_{\epsilon})
=\check F_*(\Omega_{\epsilon}),
$$
which implies
$
F(\Omega_{\epsilon})=
F_*(\Omega_{\epsilon}).
$
However, this is impossible by Proposition \ref{prop:star00}.
By Proposition~\ref{prop:03a} in the appendix,
we may assume $T$ is a non-trivial symmetry of $C$,
that is, $T$ is either a positive or negative symmetry.
If $T$ is a positive symmetry,  
Theorem \ref{prop:24} yields $T\circ F_*=F$ and
$$
F(\Omega_{\epsilon})
=T \circ F_*(\Omega_{\epsilon})=\check F(\Omega_\epsilon),
$$
contradicting \eqref{eq:check}.
So $T$ must be a negative symmetry.
\end{proof}

Similarly, the following assertion holds:

\begin{Proposition}\label{cor:00b}
Let $F\in {\mc D}_*(C)$ be a normal form.
Suppose that $C$ does not lie in any plane in $\R^3$.
Then for sufficiently small $\epsilon(>0)$,
$F(\Omega_{\epsilon})$ is
congruent to $F_*(\Omega_{\epsilon})$ if and only if $C$ has a positive symmetry. 
\end{Proposition}

\begin{proof}
If $C$ has a positive symmetry, then (1) of Theorem \ref{prop:24} implies
that $F(\Omega_{\epsilon})$ is
congruent to $F_*(\Omega_{\epsilon})$.
So it is sufficient to prove the converse.
Suppose that there exists
an isometry $T$ satisfying 
$T\circ F(\Omega_{\epsilon})
=F_*(\Omega_{\epsilon})$.
Since $F(\Omega_{\epsilon})\ne F_*(\Omega_{\epsilon})$
(cf. Proposition \ref{prop:star00}),
$T$ is a non-trivial symmetry of $C$.
If $T$ is a negative symmetry, Theorem \ref{prop:24}
yields that
$T\circ F_*=\check F$.
Then we have
$$
F(\Omega_{\epsilon})=T\circ F_*(\Omega_{\epsilon})=\check F(\Omega_\epsilon),
$$
which contradicts \eqref{eq:check}.
So $T$ must be a positive symmetry of $C$.
\end{proof}

We prove the following assertion:

\begin{Theorem}\label{thm:d4}
Let $F\in {\mc D}_*(C)$ be a normal form.
If the normalized geodesic curvature of $F$
has no self-intersections, then the number $N_F$ of
congruence classes of 
$$
F(\Omega_{\epsilon}),\quad 
\check F(\Omega_{\epsilon}),\quad 
F_*(\Omega_{\epsilon}),\quad 
\check F_*(\Omega_{\epsilon})
$$
satisfies the following properties:
\begin{enumerate} 
\item If $C$ has no symmetries and $\mu_F$ has no symmetries, 
then $N_F=4$.
\item If not the case in {\rm (1)}, then $N_F\le 2$ holds. 
\item Moreover $N_F=1$ holds if and only if
\begin{enumerate}
\item[(a)] $C$ lies in a plane and has a non-trivial symmetry, 
\item[(b)] $C$ lies in a plane and $\mu_F$ has a symmetry, or
\item[(c)] $C$ has a positive symmetry,
and $\mu_F$ has a symmetry.
\end{enumerate}
\end{enumerate}
\end{Theorem}

\begin{proof}
We suppose (1).
Since $C$ has no symmetry, $C$ does not lie in any plane.
If $N_F<4$, then replacing $F$ by $\check F,F_*$ or 
$\check F_*$, we may assume that
$F(\Omega_\epsilon)$ is congruent to 
$G(\Omega_\epsilon)$ for sufficiently small $\epsilon(>0)$, where $G$ is
$\check F,F_*$ or $\check F_*$.
If $G=\check F$, then Proposition~\ref{cor:001} implies
that $C$ has a positive symmetry, a contradiction.
If $G=F_*$, then by
Proposition \ref{cor:00b}, $C$ has a positive symmetry, a contradiction.
If $G=\check F_*$ then by Proposition \ref{cor:00a}, 
$C$ has a positive symmetry or $\Gamma$ has a
symmetry, which is also a contradiction. So we obtain (1).

We now prove (2). 
If $\mu_F$ has a symmetry,
then $N_F\le 2$ follows from Theorem~\ref{thm:2G}.
So we may assume that $C$ has a symmetry $T$.
If $T$ is trivial, then
$C$ lies in a plane and $T$ is the reflection 
with respect to the plane (cf. Lemma \ref{prop:Pi}).
Then, we have
$$
T\circ F(\Omega_\epsilon)=\check F(\Omega_\epsilon),\quad
T\circ F_*(\Omega_\epsilon)=\check F_*(\Omega_\epsilon)
$$
for sufficiently small $\epsilon>0$, and so
$N_F\le 2$ is obtained.
We next consider the case that
$T$ is non-trivial.
Then Corollary \ref{cor:24} implies that $N_F\le 2$.

Finally, we prove (3). 
We consider the case (a) or (b).
Then $C$ lies in a plane, and
$F(\Omega_\epsilon)$  (resp. $F_*(\Omega_\epsilon)$) 
is congruent to $\check F(\Omega_\epsilon)$
 (resp. $\check F_*(\Omega_\epsilon)$) by Lemma \ref{prop:Pi}.
If (a) happens, then $C$ admits a positive
symmetry by Proposition \ref{lem:03a}.
So (1) of Theorem \ref{prop:24} implies
that $F(\Omega_\epsilon)$  
is congruent to
$F_*(\Omega_\epsilon)$, and
$N_F=1$ is obtained.
If (b) happens, then 
$F(\Omega_\epsilon)=\check F_*(\Omega_\epsilon)$
by Proposition \ref{prop:ffs}. So we have $N_F=1$.

We next consider the case (c).
By (1) of Theorem \ref{prop:24},
$F(\Omega_\epsilon)$  
is congruent to
$F_*(\Omega_\epsilon)$.
Moreover, since $\mu_F$ has a symmetry,
$F(\Omega_\epsilon)$ (resp. $\check F(\Omega_\epsilon)$)
is congruent to $\check F_*(\Omega_\epsilon)$
(resp. $F_*(\Omega_\epsilon)$), so we have
$N_F=1$.

Conversely, we suppose $N_F=1$. Then
$F(\Omega_\epsilon)$ must be congruent to 
$\check F_*(\Omega_\epsilon)$, and
so, Proposition \ref{cor:00a} implies that 
\begin{itemize}
\item[(i)] $C$ has a negative symmetry, or
\item[(ii)] $\Gamma$ has a symmetry.
\end{itemize}
If $C$ lies in a plane, then (i) corresponds to (iii-a)
and (ii) corresponds to (iii-b).
So we may assume that $C$ does not lie in any plane.
Since $N_F=1$, $F(\Omega_\epsilon)$
also congruent to $\check F(\Omega_\epsilon)$.
Since $C$ is non-planar, (2) of Propositions \ref{cor:001} holds,
which implies (c).
\end{proof}

We now prove Theorem B in the introduction:

\begin{proof}[Proof of Theorem B]
We fix a curved folding
$P\in \mc P_*(\Gamma,C)$ 
arbitrarily.
Then there exists a normal form $F\in \mc D_*(C)$
such that $P=\Phi_F$.
Then $\Gamma$ gives a generator of $F$.
By Lemma~\ref{lem:L}, the condition that $\Gamma$ has a
symmetry can be replaced by the condition that
$\hat \mu_F$ has a symmetry.
By Theorem \ref{prop:introfunctor}, Theorem B is obtained 
as a corollary of Theorem \ref{thm:d4}.
\end{proof}

%\rm
As an application, we can prove the following:

\begin{Corollary}\label{cor:App}
Suppose that $C$ does not lie in any plane.
Let $f\in {\mc D}_*(C)$.
If the derivatives of
the curvature function $\kappa_f$ of $C$
and the geodesic curvature function 
$\mu_f$ of $f$
both do not vanish at the midpoint of $C$,
then 
$N_f=4$
and
the number of the congruence classes of
curved foldings in $\mathcal P_*(\Gamma,|C|)$ 
is also four, where $\Gamma$ is a generator of $f$.
\end{Corollary}

\begin{proof}
We let $l$ be the common total length of $C$ and $\Gamma$.
We denote by $\kappa$ (resp. $\mu$) 
the curvature function of $C$ (resp. $\Gamma$)
with respect to the arc-length parameter on $\mathbb I_l$.
Since $\kappa$ (resp. $\mu$) is positive-valued,
the existence of a symmetry of $C$ (resp. $\Gamma$)
implies
$$
\kappa(-s)=\kappa(s),\qquad (\text{resp.\,\,} \mu(-s)=\mu(s))
\qquad (s\in \mathbb I_l).
$$
Differentiating it at $s=0$, we have
\begin{equation}\label{eq:km}
\kappa'(0)=0,\qquad (\text{resp.\,\,} \mu'(0)=0).
\end{equation}

We may assume that $f$ is defined on a tubular neighborhood
of $I\times \{0\}$ in $I\times \R$.
In the parametrization $u\mapsto f(u,v)$ of $C$,
we suppose that
the point $s=0$ corresponds to the point $u=u_0\in I$.
Then \eqref{eq:km} is equivalent to the condition
that $d\kappa_f(u)/du$ (resp. $d\mu_f(u)/du$) 
does not vanish at $u=u_0$.
So, we obtain the conclusion.
\end{proof}

\begin{Remark}
For a given curve in $\R^2$ or $\R^3$,
it is difficult to judge that it
has a symmetry or not, in general.
So we gave Corollary \ref{cor:App}, which  can be 
considered as a practical criterion on whether 
a given developable strip (or a curved folding) 
induces mutually non-congruent isomers or not,
without use of the arc-length parametrization
of $C$. In fact, let
$f\in \mc D_*(C)$ be a developable strip
as in the proof of Corollary \ref{cor:App}.
As a preliminary step,  
one should show that the value $u_0\in I$ giving the 
midpoint $\mb c_f(u_0)$ of $C:=\mb c_f(I)$ 
lies in a certain subinterval $I_1$ in $I$.
(This subinterval is best to be taken as small as possible.)
After that if one can show that $d\kappa_f(u)/du$
and $d\mu_f(u)/du$ do not vanish 
at the same time on $I_1$,
then Corollary~\ref{cor:App} implies
that $N_f=4$.
\end{Remark}

\begin{Example}[\cite{Ori}]\label{Ex:Added901a}
Consider the space curve
$$
\mb c(s)=\left( 
\arctan s,\,\, \frac{\log(1+s^2)}{\sqrt{2}},\,\, s-\arctan s
\right)\qquad (s\in I),
$$
where $I:=[1/10,9/10]$.
We set
$
\alpha(s):={\pi(s+10)}/{24},
$
and consider the developable strip $F(s,t)$
whose first angular function is $\alpha(s)$.
It can be easily check that $F$ belongs to $\mc D_*(C)$.
Then the images of $F,\check F,F_*$ and
$\check F_*$ belongs to distinct congruence classes
as shown in \cite{Ori}.
This fact can be easily checked by applying
Corollary \ref{cor:App}.
In fact, the curvature and torsion functions of
$\mb c$ are given by
$
\kappa=\tau=\sqrt{2}/(1+s^2).
$
Since $\tau$ is non-constant, the image of
$\mb c$ does not lie in any plane.
Moreover, since $\kappa'>0$ everywhere
and
$\mu:=\kappa \cos \alpha$ satisfies
$$
\mu'(1/2)=-\frac{96 \sin \left({\pi}/{16}\right)+5 \pi  
\cos \left({\pi }/{16}\right)}{75 \sqrt{2}}<0.
$$
Thus, we have reproved the fact $N_{F}=4$.
The figures of the images of $F,\,\check F,\,F_*$ and
$\check F_*$
are given in \cite[Figure 3]{Ori}.
\end{Example}

Although, in the above example, the space curve $\mb c$
has an arc-length parametrization,
Corollary \ref{cor:App} can be applied 
without assuming $C$ is parametrized by arc-length:

\begin{Example}\label{Ex:Added901b}
We set $d\in (0,\pi/4)$
and $I:=[3\pi/8,5\pi/8]$.
Consider the embedded space curve defined by
$$
\mb c(t):=(\cos t,\sin(t+d),t) \qquad (t\in I).
$$
We let $f$ be the developable strip
whose first angular function is $\alpha:=\pi/3$ along $C:=\mb c(I)$.
Since $d\ne 0$, the torsion function of $\mb c$
is not identically equal to zero, that is, 
$C$ does not lie in any plane.
Since the curvature function $\kappa$ of $\mb c$ 
is given by
$$
\kappa(t)=\left(
\frac{3+\cos 2d+\cos2t-\cos(2t+2d)}{2(1+\sin^2t+\cos^2(t+d))^3}
\right)^{1/2},
$$
it has the following expansion
\begin{equation}\label{eq:k-series}
\kappa(t)=\frac12+\frac{\sin 2 t}2 d+o(d)
\end{equation}
with respect to the parameter $d$,
where $o(d)$ is a term of order higher than $d$.
Since the maximum of $\kappa/2(=\kappa\cos\alpha)$ on $I$
is less than the minimum of $\kappa(t)$ on $I$
for sufficiently small 
$d$ ,
we may assume that $f$ belongs to
$\mc D_*(C)$
(in fact, this holds if $d \le \pi/5$).
By \eqref{eq:k-series}, we have
$$
\kappa'(t)=(\cos 2 t) d+o(d).
$$
Since $\cos 2t$ is negative on $I$, the derivative
$\kappa'$ is negative on $I$ for sufficiently small $d$.
Since $\mu(t):=\kappa(t)/2$ is
the geodesic curvature function 
of $C$ along $f$, the derivatives
$\kappa'$ and $\mu'$ do not vanish at the
midpoint of $C$.
So, by Corollary \ref{cor:App},
the images of $f,\,\check f,\,f_*$ and
$\check f_*$ are mutually non-congruent
 for sufficiently small
$d(>0)$.
\end{Example}

In the authors' previous work,
non-congruence of
the images of $F,\,\check F,\,F_*$
and $\check F_*$ in Example
\ref{Ex:Added901a}
was shown by 
computing mean curvature functions of 
their associated developable strips.
However, this argument does not work  in general,
since the mean curvature functions of
$f,\check f,(\check f_*)^\sharp$
and $(f_*)^\sharp$ for $f\in \mc D_*(C)$
may not take distinct values along $C$
even when their images are non-congruent
each other, as follows:

\begin{Proposition}\label{prop:H}
Let $\kappa(s)$ be a $C^\infty$-function defined on an interval $I$
containing $s=0$,
and let $\alpha:I\to (0,\pi/2)$ be a $C^\infty$-function
satisfying
\begin{equation}\label{k1a1}
\kappa'(0)<0,\qquad \alpha'(0)\ge 0.
\end{equation}
Then there exists 
an embedded space curve $\mb c(s)$
defined on an interval $I_1(\subset I)$ containing $s=0$
satisfying the following properties:
\begin{enumerate}
\item $\mb c(s)$ gives an arc-length parametrization of $C:=\mb c(I_1)$
whose curvature function coincides with $\kappa$ on $I_1$.
\item Let $F(s,v)$ be the developable strip along $C$
 whose
first angular function is $\alpha(s)$.
Then $F$ belongs to $\mc D_*(C)$ for 
sufficiently small choice of $I_1$, and
the images of $F,\,\check F,\, F_*,\,\check F_*$ 
are non-congruent each other, but
\item the restriction of the
mean curvature function of $F(s,v)$ on $I_1\times \{0\}$
is equal to that of $\check F_*(-s,v)$.
\end{enumerate}
\end{Proposition}

The curvature functions of the space curves
and the first angular functions of the developable strips
given in  Example \ref{Ex:Added901b} satisfy \eqref{k1a1}.

\begin{proof}
We let $l>0$ be a number so that
$\mathbb I_l=[-l/2,l/2](\subset I)$.
Consider a $C^\infty$-function $\tau(t)$ defined on $\mathbb I_l$.
Then there exists a unique regular space curve 
$\mb c:\mathbb I_l\to \R^3$ whose curvature 
and torsion functions are $\kappa(s)$
and $\tau(s)$ ($s\in \mathbb I_l$), respectively.
Let $F(s,v)$ be the normal form  of a developable strip
whose angular function is $\alpha(s)$.
If we choose sufficiently small $l$, then
$F$ belongs to $\mc D_*(C)$ because $0<\alpha(0)<\pi/2$.
By \eqref{k1a1}, 
$
\mu:=\kappa\cos \alpha
$
satisfies 
$$
\mu'(0)=
 \kappa'(0)\cos \alpha(0)-\kappa(0) \alpha'(0) \sin \alpha(0)<0.
$$
This, with the fact $\kappa'(0)<0$,
$F$ satisfies the assumption of Corollary \ref{cor:App},
and we obtain (1) and (2).

We next adjust the torsion function $\tau$ to obtain (3):
The first angular function $A(s)$ of $F_1(s,v):=\check F_*(-s,v)$
is given by (cf. \eqref{eq:s1} and \eqref{eq:as})
\begin{equation}\label{eq:At}
A(s)=\arccos \left(\frac{\kappa(-s)\cos \alpha(-s)}{\kappa(s)}\right).
\end{equation}
We let $H(s,v)$ and $H_1(s,v)$ be the mean curvature functions
of $F(s,v)$ and $F_1(s,v)$, respectively. 
 By \eqref{eq:H}, the condition $H(s,0)=H_1(s,0)$ 
is equivalent to the identity
\begin{equation}\label{eq:Tau}
\frac{\kappa^2+(\alpha'+\tau)^2}{\sin \alpha}=
\frac{\kappa^2+(A'+\tau)^2}{\sin A},
\end{equation}
which can be considered as the quadratic equation of $\tau$.
In fact, \eqref{eq:Tau} can be written as
$$
B_0(s) \tau^2(s)+2B_1(s) \tau(s)+B_2(s)=0,
$$
where 
\begin{align*}
B_0&:=\csc \alpha-\csc A,\quad
B_1:=\alpha' \csc \alpha-A'\csc A, \\
B_2&:={\Big((\alpha')^2+\kappa^2 \sin^2 \alpha\Big) \csc \alpha-
\Big((A')^2+\kappa^2 \sin^2 A\Big)\csc A}.
\end{align*}
By \eqref{eq:At}, we have
$$
A(0)=\alpha(0), \qquad
A'(0)=\frac{2 \kappa'(0)\cot \alpha(0)}{\kappa(0)}-\alpha'(0),
$$
which imply (cf. \eqref{k1a1})
$$
B_0(0)=0,\qquad B_1(0)=\frac{-2\csc \alpha(0)}{\kappa(0)}
\Big({\kappa'(0)\cot \alpha(0)}-\kappa(0)\alpha'(0)
\Big)>0.
$$
If we set
\begin{equation}\label{eq:tau}
\tau=\frac{-B_1+\sqrt{B_1^2-B_0B_2}}{B_0}=
\frac{-B_2}{B_1+\sqrt{B_1^2-B_0B_2}},
\end{equation}
then $\dy\lim_{s\to 0}\tau(s)$ is equal to
$-B_2(0)/(2B_1(0))$. In particular, $\tau$
gives a $C^\infty$-function on $\mathbb I_l$ for sufficiently small $l$.
So  (3) is obtained.
\end{proof}

\section{The case that $C$ is closed}
In this section, we consider the case $J=\mathbb T^1_l$, that is,
$C$ is a closed embedded curve in $\R^3$.
We fix an immersed curve $\gamma:\R\to \R^2$ with
arc-length parameter whose
curvature function $\mu(s)$ of $\gamma(s)$ 
is an $l$-periodic smooth function $\tilde \mu:\R\to \R$
satisfying
$$
\dy\max_{s\in [0,l]}|\mu(s)|< \dy\min_{w\in \mathbb T^1_l}\kappa(w).
$$
We let $\Gamma (\subset \R^2)$ 
be the image $\gamma([b_0,b_0+l])$  
of the plane curve $\gamma$ for some $b_0\in [0,l)$
such that $\gamma([b_0,b_0+l))$ has no self-intersections.
Here, such a $\Gamma$ is not uniquely determined,
since the possibilities for $\Gamma$ have the ambiguity of the choice of 
the value $b_0$
($\Gamma$ and $C$ have the same length).

If $P_*(\Gamma,C)$ is defined as in the introduction,
then, for each $P\in \mc P_*(\Gamma,C)$,
we can write $P=\Phi(F)$ for a developable strip
$F\in \mc D_*(C)$ written in a normal form
such that $\mu(s)$
coincides with the lift $\tilde \mu_F$
of the normalized curvature function 
$\mu_F$ of $F$ (see the beginning of Section~2).
 
\begin{Remark}\label{rem:Gamma}
The plane curve
$\Gamma:=\gamma([b_0,b_0+l])$ 
is prepared as the crease pattern
of curved foldings belonging to $\mc P_*(\Gamma,C)$.
So $\Gamma$ should not have any self-intersection.
However, this property depends on the choice of $b_0\in [0,l)$.
In fact, we consider the trochoid $\gamma$ defined by
\begin{equation}\label{eq:Cyc}
\gamma(t):=\left(2t/3-\sin t, 1-\cos t\right)
\qquad (t\in \R).
\end{equation}
Then  its curvature function is $2\pi$-periodic, and 
two subarcs $\Gamma_1=\gamma([0,2\pi])$
and $\Gamma_2=\gamma([\pi,3\pi])$
can be considered as fundamental pieces of 
the trochoid (cf. Figure~\ref{fig:SI}).
The arc $\Gamma_2$ has a self-intersection
but $\Gamma_1$ does not.
\end{Remark}

\begin{figure}[thb]
\begin{center}
\includegraphics[height=2cm]{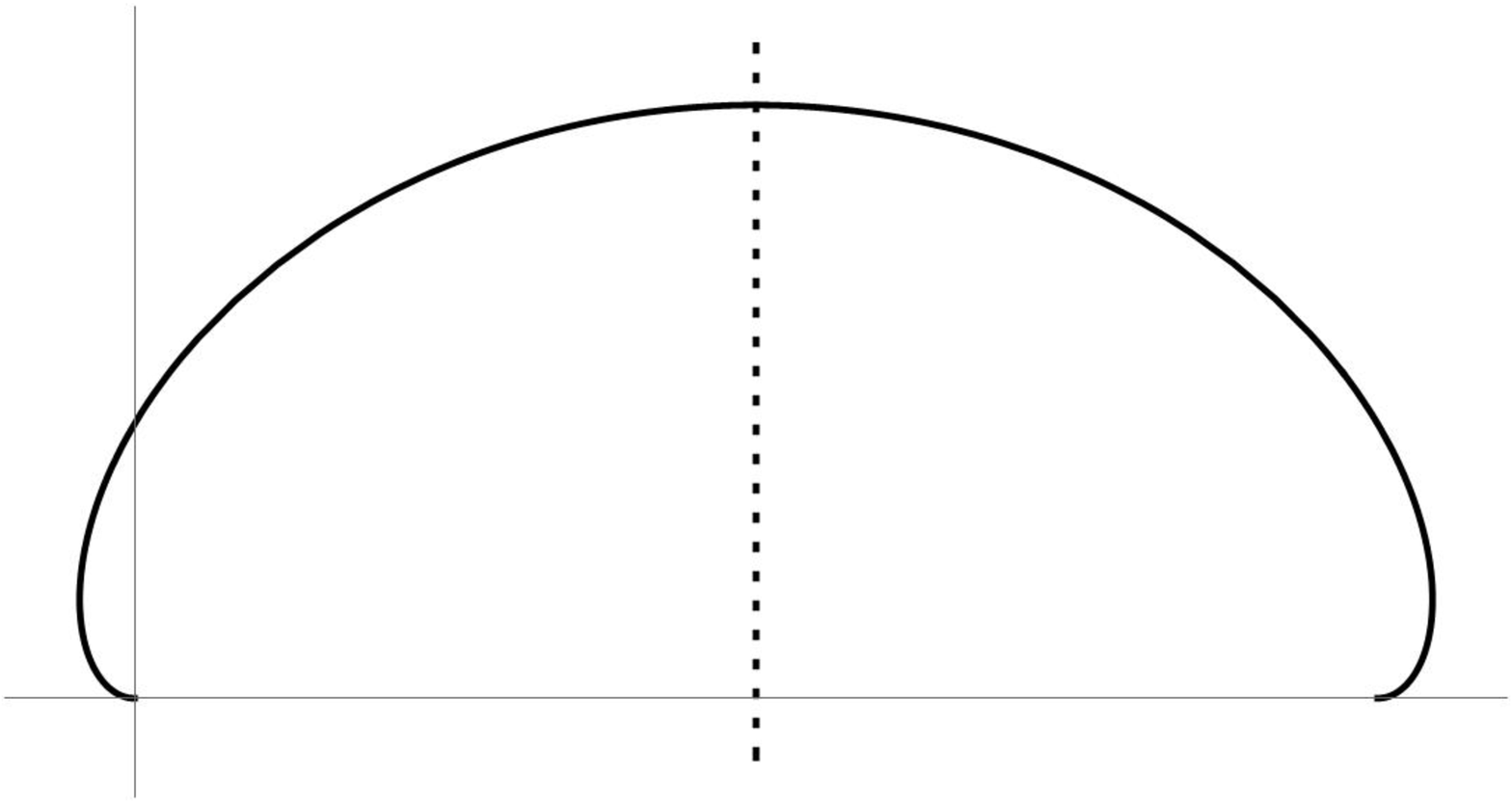}\qquad \quad
\includegraphics[height=2cm]{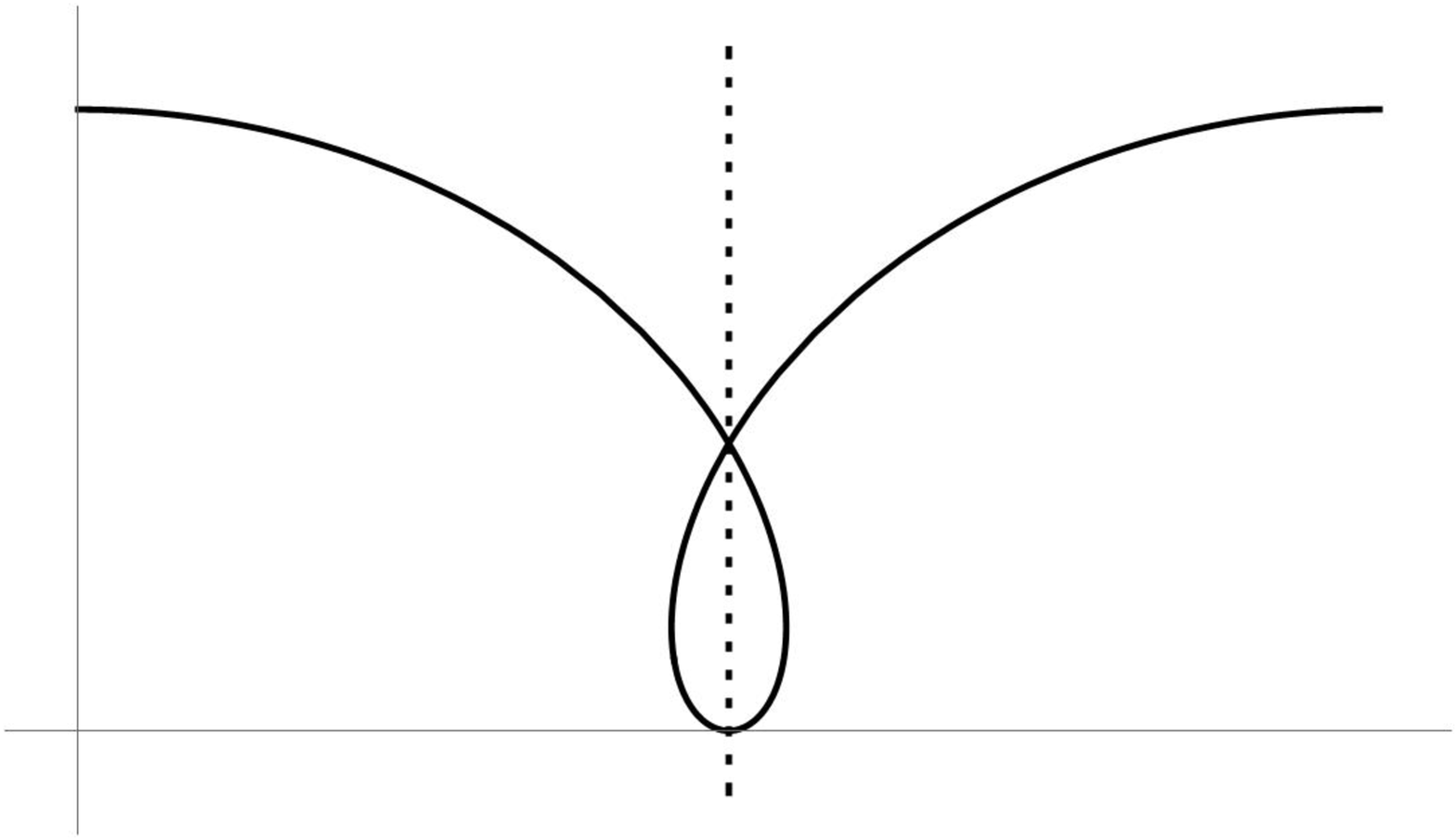}
\caption{The shape of $\Gamma_1$ (left) and
$\Gamma_2$ (right)}\label{fig:SI}
\end{center}
\end{figure}

We fix a normal form $F\in \mc D_*(C)$
such that $\Gamma$ is a generator of $F$.
Without loss of generality, we may assume that $F$ is defined on 
a tubular neighborhood of $\mathbb T^1_l\times \{0\}$
in $\mathbb T^1_l\times \R$.
We assume that the geodesic curvature function
$\mu_F$ of $F$ coincides with $\mu$.
For each fixed $b\in [0,l)$ and $\sigma\in \{1,-1\}$, we set
$$
\tilde{\mb c}(s):=\mb c(\sigma s+b),
$$
which parametrizes the curve $C:=\mb c(\mathbb T^1_l)$.
Since $F\in \mc D_*(C)$, we can apply
Proposition \ref{prop:Equi} for the curve $\tilde C:=\tilde {\mb c}(\mathbb T^1_l)$.
Since $\sigma=\pm 1$, we obtain 
normal forms $F_b^i\in \mc D_*(C)$ ($i=1,2,3,4$) 
whose first angular functions 
$\alpha^i_{b}\in C^\infty_{\pi/2}(\mathbb T^1_l)$ 
satisfy
$\alpha^1_b=-\alpha^2_b>0$, $\alpha^3_b=-\alpha^4_b>0$
and
\begin{align}\label{eq:Alpha}
& (\mu_{F_b^1}=)\kappa(s+b)\cos \alpha^1_{b}(s)=\mu_F(s),\quad
(\mu_{F_b^2}=)\kappa(s+b)\cos \alpha^2_{b}(s)=\mu_F(s), \\
\nonumber
& (\mu_{F_b^3}=)\kappa(-s+b)\cos \alpha^3_{b}(s)=\mu_F(s),\quad
(\mu_{F_b^4}=)\kappa(-s+b)\cos \alpha^4_{b}(s)=\mu_F(s),
\end{align}
where $\kappa(s)$ is the curvature function of
$\mb c(s)$ ($s\in \mathbb T_l$).
By definition, these four developable strips are all geodesically 
equivalent to $F$, and satisfy
$$
F_b^1(s,0)=F_b^2(s,0)=\mb c(s+b),
\qquad F_b^3(s,0)=F_b^4(s,0)=\mb c(-s+b)
$$
and $F^1_0=F$.
The following assertion holds:

\begin{Proposition}\label{prop:(a)}
Let $G\in \mc D_*(|C|)$ be a
normal form. If $G$ is geodesically equivalent to $F$,
then there exist $b\in [0,l)$ and $j\in \{1,2,3,4\}$
such that $G=F_b^j$. 
\end{Proposition}

\begin{proof}
Since $G$ is geodesically equivalent to $F$,
$
\mu_G(\sigma s+b)=\mu_F(s)
$
holds for some $\sigma\in \{1,-1\}$ and $b\in \mathbb T^1_l$.
Replacing $F$ by $F^\sharp$ if necessary, we may assume that
$\sigma=1$. We set $G_1(s,v):=G(s+b,v)$.
Then we have $\mu_{G_1}(s,0)=\mu_F(s,0)$ and $F,G\in \mc D_*(C)$.
By Proposition \ref{fact:f},
$G_1$ must coincide with $F$ or $\check F$,
which proves the assertion because of
the definition of the families $\{F_{b}^j\}$. 
\end{proof}

\begin{Corollary}\label{prop:(a2)}
Let $G\in \mc D_*(|C|)$ be a
normal form. If $G(\Omega_{\epsilon})$ 
is congruent to $F(\Omega_{\epsilon})$ for sufficiently small 
$\epsilon(>0)$, then there exist an isometry $T$ of $\R^3$ and
$b\in [0,l)$ and $\sigma\in \{1,-1\}$
such that $G=T\circ F_b^j$. 
\end{Corollary}

\begin{proof}
Since $G(\Omega_{\epsilon})$ is congruent to $F(\Omega_{\epsilon})$,
there exists an isometry $T$ of $\R^3$ such that
$G(\Omega_{\epsilon})=T\circ F(\Omega_{\epsilon})$.
By Propositions \ref{prop:U} and \ref{L0},
$T\circ F$ is right equivalent to $G$.
In particular, there exists a diffeomorphism
$\phi$ such that $T\circ F\circ \phi=G$.
Since $s\mapsto T\circ F\circ \phi(s,0)$ and $s\mapsto G(s,0)$
are both arc-length parametrizations,
there exist $b\in [0,a)$ and $\sigma\in \{1,-1\}$
such that $\phi(s,0)=\sigma s+b$.
In particular, $T\circ F$
is geodesically equivalent to $G$,
and so we obtain the conclusion applying
Proposition~\ref{prop:(a)}.
\end{proof}

We prove the following assertion:

\begin{Theorem}\label{thm:48}
Let $F\in \mc D_*(C)$ and
$
\{F^i_{b}\}_{b\in \mathbb T^1_l}
$
$(i=1,2,3,4)$ 
be defined as above.
Then the following assertions hold:
\begin{enumerate}
\item[(a)]
If $g\in \mc D_*(|C|)$ is geodesically equivalent to $F$, then
there exist $b\in [-l/2,l/2)$ and $j\in \{1,2,3,4\}$
such that
$g$ is right equivalent to $F_b^j$.
\item[(b)] Suppose that $C$ is not a circle and 
$\mu_F$ is not a constant function. 
Then, for each 
$\{i,b\}_{(i,b)\in \{1,2,3,4\}\times \mathbb T^1_l}$, 
the set
\begin{align*}
\Lambda^i_{b}&:=\Big\{(j,c)\in \{1,2,3,4\}\times \mathbb T^1_l \,;\, \\
& \phantom{aaaaa} \text{$F^j_{c}(\Omega_{\epsilon})$ 
is congruent to $F_b^i(\Omega_{\epsilon})$ for sufficiently 
small $\epsilon(>0)$}\Big\}
\end{align*}
is finite. 
\item[(c)] Suppose that $C$ and 
$\mu_F$ have no symmetries, 
then for each $i\in \{1,2,3,4\}$ and 
$b\in [-l/2,l/2)$, 
the set $\Lambda^i_{b}$ consists of a single element.
\end{enumerate}
\end{Theorem}

\begin{proof}
Although the strategy of the proof is essentially the same as 
the corresponding proof for cuspidal edges given in 
\cite{HNSUY2}, we need several modifications, since
the real analyticity of $C$ is assumed in \cite{HNSUY2} and 
the corresponding proof for cuspidal edges
does not apply the properties of their normal forms.

The first assertion (a) follows from
Proposition \ref{prop:(a)}. So we prove (b).
Without loss of generality, we may set $F=F_b^i$.
Suppose that the congruence classes of $\Lambda^i_{b}$ are not finite. 
By Proposition \ref{prop:(a)},
there exist a sequence $F_n:=F^{j_n}_{b_n}$
($n\ge 1$) in $\mc D_*(|C|)$ such that 
\begin{itemize}
\item $\{(j_n,b_{n})\}_{n=1}^\infty$ consisting of  
mutually distinct elements in $\{1,2,3,4\}\times \mathbb T^1_l$, and
\item $F_n(\Omega_{\epsilon})$ ($n\ge 1$)
are all congruent to $F(\Omega_{\epsilon})$.
\end{itemize}
Taking a subsequence of $\{(j_n,b_{n})\}_{n=1}^\infty$
if necessary, we may assume that 
\begin{enumerate}
\item[(A1)] $j:=j_n$ does not depend
on $n$, and 
\item[(A2)] $\{b_{n}\}_{n=1}^\infty$ consists of
mutually distinct numbers. 
\end{enumerate}
By Corollary \ref{prop:(a2)}, for each $n$,
there exist an isometry 
$T_n$ of $\R^3$, $d_n\in [0,l)$ and 
$\sigma_n\in \{1,-1\}$ such that
$F_n(s,v)=T_n\circ F(\sigma_n s+d_n, v)$.
Then we have the identity
\begin{equation}\label{eq:fn2}
\mu(s)=\mu(\sigma_n s+d_n)\qquad (s\in \mathbb T^1_l).
\end{equation}
Since the number of possibilities for the signs $\sigma_n$ 
is at most two, we may set $\sigma:=\sigma_n$.
Since $\mu$ is not constant, 
we can take $s_0\in [-l/2,l/2)$ such that $\mu(s_0)$ 
is a regular value of the function $\mu$ on $\mathbb T^1_l$.
Substituting $s=s_0$, \eqref{eq:fn2}
yields
$
\mu(s_0)=\mu(\sigma s_0+d_n).
$
Since $\mathbb T^1_l$ is compact and $\mu(u_0)$ is a regular value of $\mu$, 
the possibilities for $\{d_n\}_{n=1}^\infty$ are finite, and we may assume
$d:=d_n$ does not depend on $n$.
So we have
$$ 
F(\sigma' s+b_n,v)=F_n(s,v)=T_n\circ F(\sigma s+d, v)
$$
for $n\ge 1$, where $\sigma'=-1$ if $j\in \{1,2\}$ and
$\sigma'=+1$ if $j\in \{3,4\}$.
Substituting $v=0$, we have
$
F_n(\sigma' s+b_n,0)=T_n\circ F(\sigma s+d, 0).
$
Since $C$ is not a circle, the number of symmetries of $C$ is 
finite (see Proposition~\ref{prop:Finite} in the appendix).
So, the number of possibilities for $b_n$ and $T_n$ are also finite,
contradicting (A2), and we obtain (b).

Finally, we prove (c). If the assertion fails, 
then we can choose $(j,c) \in \Lambda$.
By Corollary \ref{prop:(a2)},
there exist
$T$ of $\R^3$, $b\in [0,l)$ and $\sigma\in \{1,-1\}$
such that
\begin{equation}\label{eq:F4TF}
F^j_c(s,v)=T\circ F(\sigma s+b,v).
\end{equation} 
Substituting $v=0$, 
it can be easily check that $T(C)=C$.
Since $C$ admits no symmetries, $T$
must be the identity map. So \eqref{eq:F4TF}
reduces to
$
F^j_c(s,v)=F(\sigma s+b,v).
$
Since $\mu_F$ coincides with $\mu_{F^j_c}$ 
(cf. \eqref{eq:Alpha}),
we have
$
\mu_F(s)=\mu_F(\sigma s+b)\,\, (s\in \mathbb T^1_l).
$
Since $\mu_F$ has no symmetries,
we have $(\sigma,b)=(1,0)$, that is, $F^j_c=F$ holds,
proving the assertion (c).
\end{proof}

\begin{proof}[Proof of Theorem C]
Like as in the case of Theorem B,
Theorem C is obtained by applying 
Theorem \ref{thm:48} and
Theorem \ref{prop:introfunctor}.
Condition (3) is not used to prove Theorem C, which is
required for the definition of the set $\mc P_*(\Gamma,|C|)$. 
\end{proof}

\medskip
From the viewpoint of curved foldings,
the most interesting case is that the crease pattern 
$\Gamma$ is a closed curve without self-intersections as well as the crease $C$:

\begin{Example}\label{ex:cm}
Consider a curve 
$$
\mb c_m(t):=\Big( (2+\cos m t)\cos t, (2+\cos m t)\sin t,\sin m t\Big)
$$
lying on a rotationally symmetric torus in $\R^3$.  
We denote by $L_m$ the length of $\mb c_m$.
Then for each $m\ge 2$, the inequality 
$\dy\min_{t\in [0,2\pi]}\kappa_m(t)>2\pi/L_m$,
holds, where $\kappa_m(t)$ is the curvature of $\mb c_m(t)$.
So if $\Gamma$ is an ellipse of length $L_m$
which is sufficiently close to
a circle, then ${\mc P}_*(\Gamma,C_m)$ is non-empty.
For example, consider an ellipse
$
\gamma(t):=(\cos t,a \sin t)
$
for $a:=6/5$,
whose curvature function is given by 
$
\mu(t)=a(\sin^2 t+a^2 \cos^2 t)^{-3/2}.
$
We set $m=3$ and
$$
\tilde \gamma(t):=k \gamma(t),\qquad 
k:=L_3/l^{}_\gamma \approx 3.31,
$$
where $l^{}_\gamma$ is the length of $\gamma$.
Then the two curves $\tilde \gamma$ and $\mb c_3$
have a common length.
We reparametrize $\tilde\gamma$ by the arc-length parameter 
$s$ ($0\le s\le L_3$)
so that $\tilde\gamma(0)=(k,0)$. 
Let $F\in {\mc D}_*(C_3)$ be the developable
strip along $C_3$ such that $\mu_F(t)=\hat \mu(t)$, where
$
\hat \mu(s)
$
is the curvature function of $\hat \gamma(s):=\tilde \gamma(t(s))$.
Figure \ref{fig:closed} indicates the crease patterns of
the curved foldings induced 
by $F$ (left), $F^1_{b}$ (center) and 
$F^1_{2b}$ (right) for $b=1/8$ 
via 
the map $\Phi$, respectively (cf. \eqref{eq:Alpha} and
  Definition \ref{def:p-ps}\rm). 
\end{Example}

\begin{figure}[htb]
\begin{center}
\includegraphics[height=3.9cm]{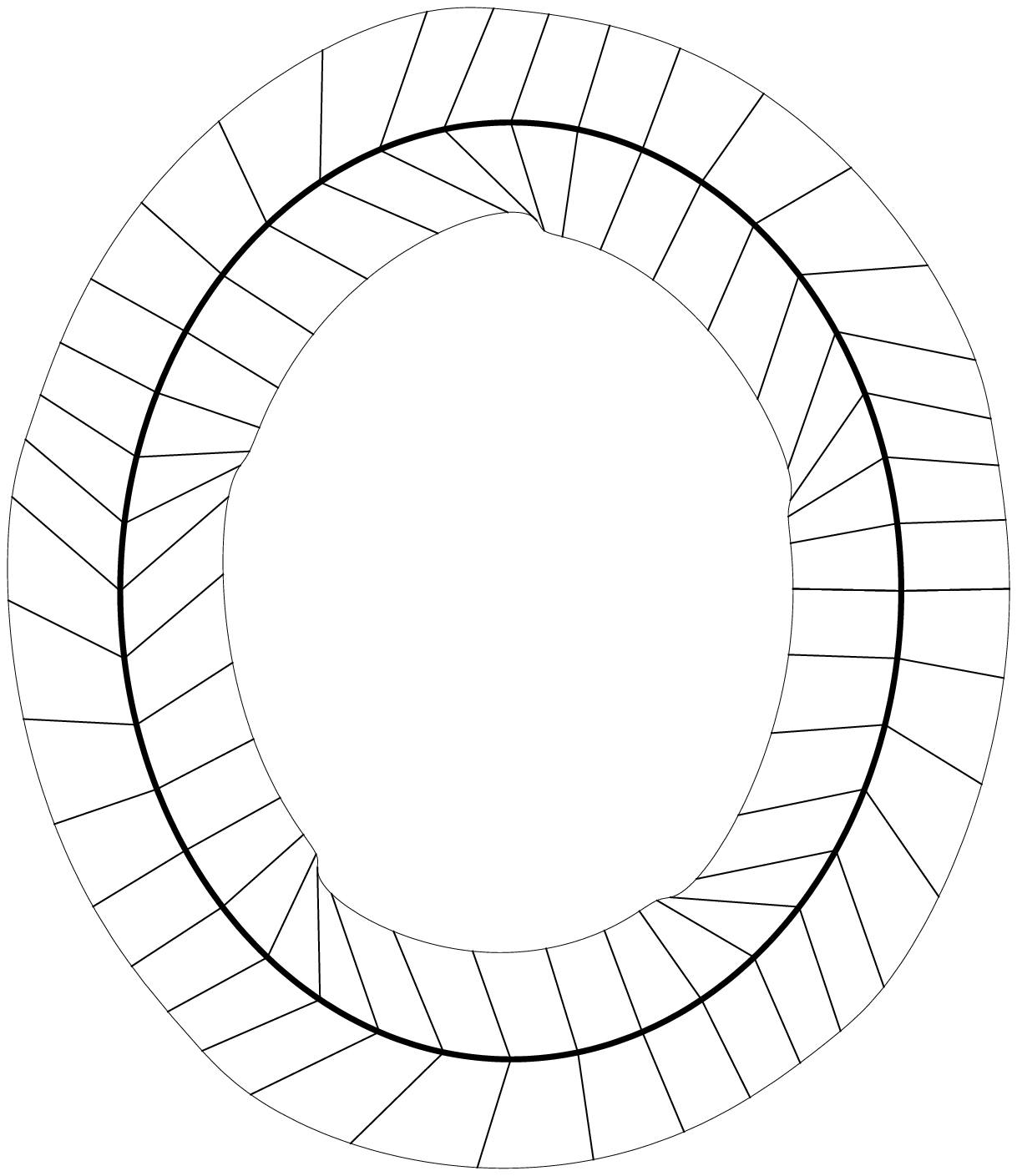}\,\,\,\,\,
\includegraphics[height=3.9cm]{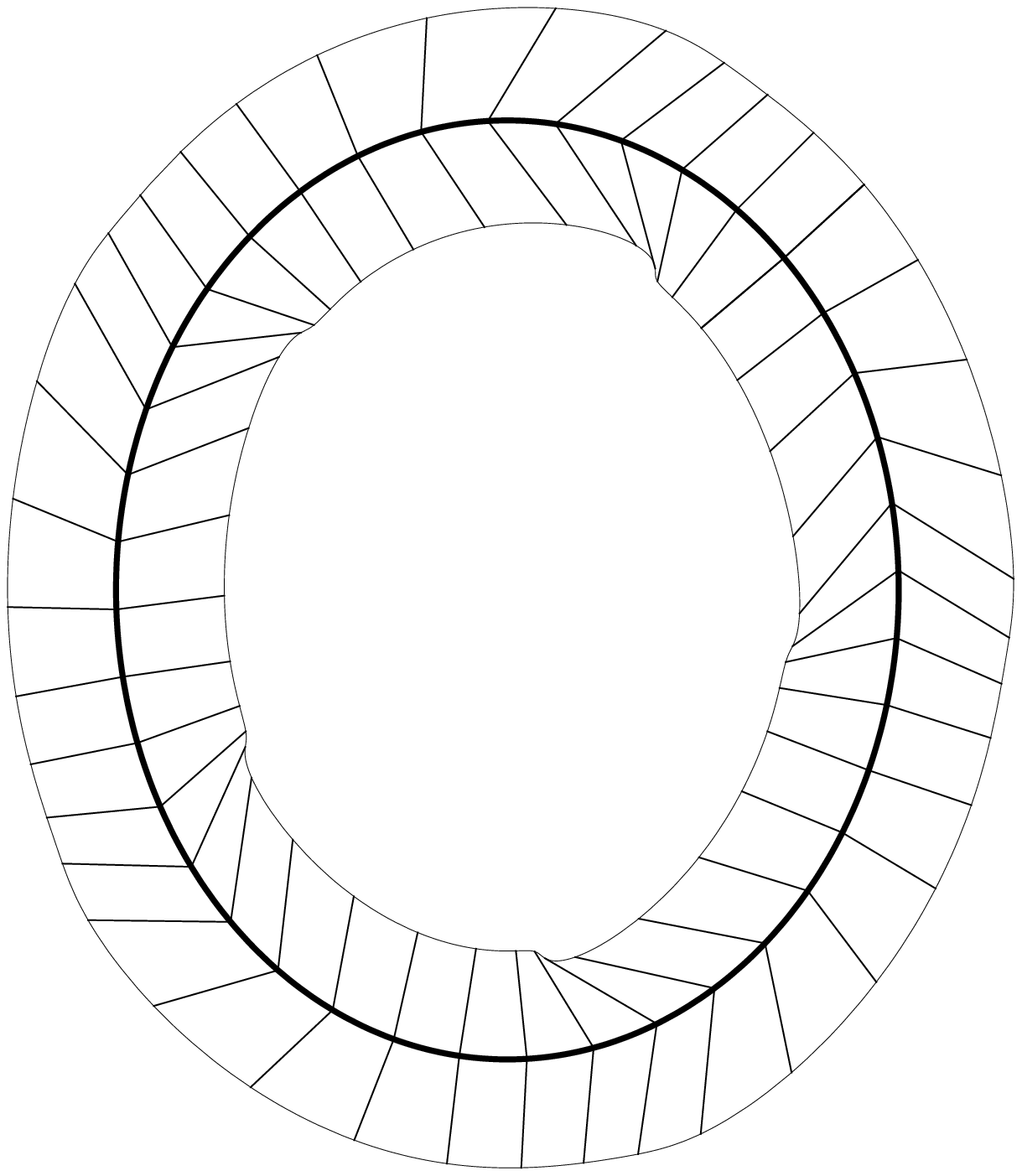}\,\,\,\,\,
\includegraphics[height=3.9cm]{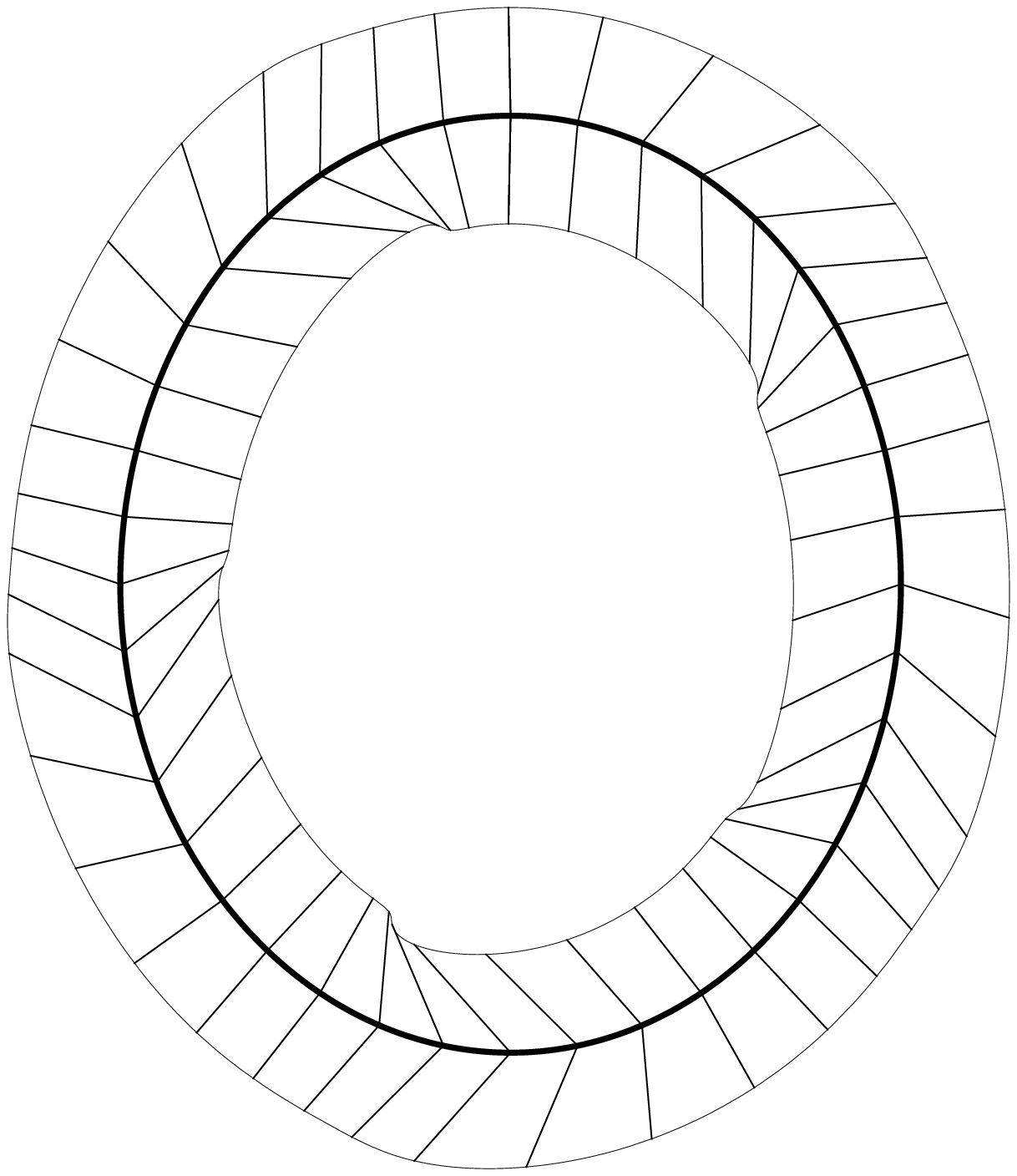}
\caption{The ruling directions on the ellipse 
of three non-congruent curved folding along $C_3$.}
\label{fig:closed}
\end{center}
\end{figure}

\appendix
\section{Symmetries of curves}

In this appendix, we prove several assertions
on symmetries of curves in $\R^2$ or $\R^3$:
Let $\mb c:\mathbb I_l\to \R^3$ be an embedded arc parametrized by
arc-length, and denote by $C$ its image, where $\mathbb I_l=[-l/2,l/2]$.
Let $T$ be a symmetry (cf. Definition \ref{def:S})
 of $C$ which reverses the orientation of $C$.
Here, we ignore the orientation of $C$.
The following assertion holds obviously.

\begin{Proposition}\label{prop:03a}
Supose the space curve $C$ admits a symmetry $T$.
Then $T$ is an involution, and
has the following properties:
\begin{enumerate}
\item $C$ lies in a plane, and $T$ is
the reflection with respect to the plane, or
\item $T$ is a positive or negative non-trivial symmetry of $C$.
\end{enumerate}
\end{Proposition}

If $T$ be a non-trivial symmetry of $C$, then
$T\circ \mb c(-s)=\mb c(s)$ holds.
By Differentiating, we have
\begin{align}\label{eq:TNB}
&-T\mb e(-s)=\mb e(s),\quad \kappa(-s)=\kappa(s),\quad
T\mb n(-s)=\mb n(s),\\
\nonumber
&-T\mb b(-s)=\sigma \mb b(s),\quad
\tau(-s)=\sigma \tau(s),
\end{align}
where $\sigma\in \{+,-\}$ is the sign of $\det(T)$.
Using this, we show the following assertion:

\begin{Proposition}\label{lem:03a}
The space curve $C$ admits a positive symmetry 
and a negative symmetry $($cf. Definition \ref{def:S}$)$
at the same time if and only if $C$ lies in a plane and
$C$ has a non-trivial symmetry.
\end{Proposition}

\begin{proof}
Suppose that $C$ admits a positive symmetry $T_1$
and a negative symmetry $T_2$ at the same time.
By \eqref{eq:TNB},
the torsion function $\tau(s)$ of $\mb c(s)$
satisfies $(-1)^{i+1}\tau(-s)=\tau(s)$.
So $\tau$ must vanish identically, that is, $C$
lies in a plane.
Since the trivial symmetry $C$ is 
uniquely determined
and is the reflection with respect to $\Pi$,
the symmetry $T_1$ of $C$ must be non-trivial, and 
the \lq \lq only if" part is proved.

To prove the converse, we suppose
that $C$ lies in a plane and
$C$ has a non-trivial symmetry.
Then each point of $C$ is fixed by
the reflection $T_0$ with respect to the plane.
We let $T_1$ be a non-trivial symmetry of $C$.
If $T_1$ is a positive (resp. negative) symmetry of $C$, 
then $T_2:=T_0\circ T_1$ is a negative (resp. positive) symmetry of $C$.
So we obtain the conclusion.
\end{proof}

\begin{Proposition}\label{prop:Finite}
{\it Let $\mb c:\mathbb T^1_l \to \R^3$ $(l>0)$
be an embedded $C^\infty$-regular curve.
Then  
$C:=\mb c(\mathbb T^1_l)$
has at most finitely many symmetries, unless $C$ is a circle}.
\end{Proposition}

\begin{proof}
We may assume that $\mb c(s)$ has an arc-length parametrization
and denote by $\kappa(s)$ its curvature function.
Since $\mb c $ is not a circle, $\kappa$ is not a constant function.
By Sard's theorem, the set of critical values of $\kappa$ is of measure zero.
We choose a regular value $r$ of the curvature function 
$\kappa$
such that $\kappa^{-1}(r)$ 
consists of finitely many points $s_1,\dots,s_n \in \mathbb T^1_l$.
Then we set $\mb x_i=\mb c(s_i)$ ($i=1,\dots,n$).
Suppose that there are infinitely many distinct non-trivial symmetries
$\{T_j\}_{j=0}^\infty$ of $C$.
Then
$
T_j(\mb x_1)\in \{\mb x_1,\dots,\mb x_n\}.
$
So, there exists $j_0\in \{1,\dots,n\}$ such that
$
T_j(\mb x_1)=\mb x_{j_0}=T_1(\mb x_1)
$
for infinitely many $j\ge 2$. 
Since there exist at most finitely many
non-trivial symmetries of $C$ fixing $\mb x_1$,
$\{T_j\}$ must consist of at most finitely many isometries of $\R^3$
a contradiction.
\end{proof}

\begin{acknowledgments}
The authors thank Professors Jun Mitani 
and Wayne Rossman for valuable comments.
\end{acknowledgments}

\end{document}